\documentclass[11pt]{article}
\usepackage{amsmath,amsfonts,amssymb,amsthm}
\usepackage{mathrsfs}
\usepackage{latexsym}
\usepackage[english]{babel}

\usepackage{graphicx}
\usepackage[usenames,dvipsnames]{xcolor}

\renewenvironment{proof}[1][\proofname]{{\bfseries #1.} }{\qed}

\def\Cov{{\rm Cov\,}}

\newcommand{\field}[1]{\mathbb{#1}}

\newcommand{\R}{\field{R}}

\newcommand{\N}{\field{N}}

\newcommand{\Var}{{\rm Var}}

\newcommand{\F}{{\mathscr{F}}}

\newcommand{\goto}{{\longrightarrow}}

\def\authors#1{{ \begin{center} #1 \vspace{0pt} \end{center} } \smallskip}
\def\institution#1{{\sl \begin{center} #1 \vspace{0pt} \end{center} } }

\def\title#1{{\huge\bf  \begin{center} #1 \vspace{0pt} \end{center}  } \smallskip}

\def\E{{\mathbb{ E}}}

\def\P{{\mathbb{P}}}
\def\F{{\mathscr{F}}}

\def\paref#1{(\ref{#1})}

\newtheorem{theorem}{Theorem}[section]
\newtheorem{proposition}[theorem]{Proposition}
\newtheorem{lemma}[theorem]{Lemma}

\begin{document}

\title{\sc{The Defect of \\Random Hyperspherical Harmonics}}
\date{Sep 28, 2015}
\authors{Maurizia Rossi
}
\institution{
MAP5-UMR CNRS 8145, Universit\'e Paris Descartes, France \\
E-mail address: \texttt{maurizia.rossi@parisdescartes.fr}
}

\begin{abstract} 
Random hyperspherical harmonics are Gaussian Laplace eigenfunctions on the unit $d$-sphere ($d\ge 2$). We investigate the distribution of their defect i.e., the difference between the measure of positive and negative regions.  
Marinucci and Wigman studied the two-dimensional case giving the asymptotic variance \cite{def} and a Central Limit Theorem \cite{Nonlin}, both in the high-energy limit.

Our main results concern asymptotics for the defect variance and  quantitative CLTs in Wasserstein distance, in any dimension. 
The proofs are based on Wiener-It\^o chaos expansions for the defect, a careful use of asymptotic results for all order moments of Gegenbauer polynomials and Stein-Malliavin approximation techniques by Nourdin and Peccati \cite{nou-pe2, noupebook}. Our argument requires some novel technical results of independent interest that involve integrals of the product of three hyperspherical harmonics.

\smallskip


\smallskip

\end{abstract}

\section{Introduction}\label{intro}

Let $f:\mathcal M \to \mathbb R$ be any real-valued function defined on some compact Riemannian manifold $(\mathcal M,g)$ and let $\mu_g$ denote the induced measure on $\mathcal M$. 
The defect $D(f)$ of $f$ is the difference between the measure of ``hot'' and ``cold'' regions:
\begin{equation*}
D(f) := \mu_g(f^{-1}(0,+\infty)) - \mu_g(f^{-1}(-\infty,0)).
\end{equation*}
We can hence write
\begin{equation}\label{def}
D(f)=\int_{\mathcal M} \mathcal H(f(x))\,d\mu_g(x),
\end{equation}
where $\mathcal H$ denotes the Heaviside function $\mathcal H(t) := 1_{(0,+\infty)}(t) - 1_{(-\infty,0)}(t)$, $t\in \mathbb R$.

An important case is where $f$ is a Laplacian eigenfunction. We recall that a function $f$ is called a Laplacian eigenfunction if it is a non-trivial solution of the Schr\"odinger equation 
\begin{equation}\label{eqSch}
\Delta_{g} f + E f =0,
\end{equation}
where $\Delta_g$ stands for the Laplace-Beltrami operator on $(\mathcal M,g)$ and  $E> 0$.  It is well-known that the (purely discrete) spectrum of $-\Delta_g$ consists of a non-decreasing sequence of positive eigenvalues whose corresponding sequence of eigenfunctions forms a complete orthonormal basis for $L^2(\mathcal M)$, the space of square integrale functions on the manifold. Observe that we allow multiple eigenvalues i.e., spectral degeneracies. 

An increasing amount of mathematics research has focused on the geometry of the nodal set   $f^{-1}(0)$ (see e.g. \cite{bruning, bruning-gromes, D-F, yau}) and 
its complement $\mathcal M\setminus f^{-1}(0)$ (see e.g. \cite{GRS,JZ2}), associated with Laplacian eigenfunctions $f$. Note that nodal sets are customarly called ``nodal lines" in the two-dimensional case -- being, generically, smooth curves -- and also that the connected components of $\mathcal M\setminus f^{-1}(0)$ are often referred to as ``nodal domains". 
The defect \paref{def} is one of the most natural  functionals \cite{def} associated with the geometry of the latter. 

Recently, a growing interest has been attracted by \emph{random} eigenfunctions on manifolds (see also \cite{meckes}) - especially on the two-dimensional sphere and the standard flat torus (e.g. \cite{BMW17, AmP,  misto, MRW17, def, Nonlin, nazarovsodin, RW, Wig}). In the latter references, the space of eigenfunctions is endowed with some probability measure and the geometry of their (random) nodal sets and domains is studied (in the high-energy limit, i.e. when the magnitude of the eigenvalues diverges to infinity). 
See \S \ref{det vs ran} and \cite{wigsurvey} for further discussions. 
In this paper we study the high-energy behavior of the defect of random Laplacian eigenfunctions on hyperspheres. 

\smallskip

\emph{\large Some conventions.} In this manuscript, given two sequences $a_n, b_n$ of positive numbers, we will write 
$a_n\sim b_n$ if $\lim_{n\to +\infty} a_n/b_n=1$, whereas $a_n = O(b_n)$ or equivalently $a_n\ll b_n$ (resp. $a_n=o(b_n)$) if $a_n/b_n$ is asymptotically bounded (resp. $a_n/b_n \to 0$). Finally, $a_n\asymp b_n$ will mean that  $a_n/b_n\to c$, for some $c>0$. 
Every random object will be defined on the same probability space $(\Omega, \F, \P)$, $\E$ shall denote the expectation under the measure $\P$ and, as usual, $\mathop{\to}^{\mathcal L}$  convergence in distribution whereas $\mathop{=}^{\mathcal L}$ equality in law. 

\smallskip


\subsection{Random Hyperspherical Harmonics}\label{introH}

We deal with the case $\mathcal M = \mathbb S^d\hookrightarrow \mathbb R^{d+1}$, the unit $d$-dimensional sphere with the natural metric ($d\ge 2$). The induced measure is the Lebeasgue measure $dx$. The eigenvalues of the Laplace-Beltrami  operator on $\mathbb S^d$ (which will be denoted by $\Delta_d$ from now on),  are of the form $-\ell(\ell+d-1)$, for $\ell\in \mathbb N$, and the dimension $n_{\ell;d}$ of the $\ell$-th eigenspace is $$n_{\ell;d}=\frac{2\ell +d-1}{\ell}{\ell+d-2 \choose \ell -1}\sim \frac{2}{(d-1)!}\ell^{d-1},\qquad \text{as } \ell\to +\infty.$$ 
An orthonormal basis for the $\ell$-th eigenspace is given by the family of (real-valued) hyperspherical harmonics $(Y_{\ell,m;d})_{m=1}^{n_{\ell;d}}$ (see e.g. \cite[\S 9.3]{vilenkin2})
$$
\Delta_d Y_{\ell,m;d} +\ell(\ell+d-1) Y_{\ell,m;d} =0. 
$$

We now endow the $\ell$-th eigenspace with a Gaussian measure, i.e. we consider  the  $\ell$-th (real-valued) random eigenfunction $T_\ell:=T_{\ell;d}$ on $\mathbb S^d$ to be defined as 
\begin{equation}\label{rf}
T_\ell(x) := \sum_{m=1}^{n_{\ell;d}} a_{\ell,m;d} Y_{\ell,m;d}(x),\qquad x\in \mathbb S^d,
\end{equation}
where $\left (a_{\ell,m;d}\right )_{m=1}^{n_{\ell;d}}$ are i.i.d. centered Gaussian random variables with variance given by
\begin{equation}\label{var_aelle}
\Var(a_{\ell,m;d}) = \frac{|\mathbb S^d|}{n_{\ell;d}},
\end{equation}
$|\mathbb S^d|$ denoting the (Lebeasgue) measure of the hyperspherical surface. Equivalently, we can define $T_\ell$ as the isotropic centered Gaussian random field on $\mathbb S^d$ whose covariance kernel is 
\begin{equation}\label{cov}
\Cov(T_\ell(x), T_\ell(y)) = G_{\ell;d}(\cos d(x,y)),\quad x,y\in \mathbb S^d,
\end{equation}
where $G_{\ell;d}$ stands for the normalized $\ell$-th Gegenbauer polynomial \cite[\S 4.7]{szego} and $d(x,y)$ denotes the (spherical) geodesic distance between $x$ and $y$.  

To be more precise, $G_{\ell;d}=\alpha_\ell^{-1} P_\ell^{(d/2-1,d/2-1)}$, where $\left( P_\ell^{(a,b)} \right)_\ell$ denotes the family of Jacobi polynomials\footnote{Recall that $\left( P_\ell^{(a,b)} \right)_\ell$ is a family of orthogonal polynomials on the interval $[-1,1]$ with respect to the weight $
(1-t)^{a}(1+t)^b
$.  } \cite[Chapter 4]{szego} and $\alpha_\ell={\ell + d/2 -1 \choose \ell}$ is a normalizing factor. It turns hence out that $T_\ell(x)$ has unit variance for every $x\in \mathbb S^d$. 

This model was studied  in \cite{maudom} and, in the particular case $d=2$ in  \cite{BMW17, eulvale,fluct,dogiocam,MRW17,def,Nonlin,nazarovsodin,Wig} e.g.
Note that when $d=2$, \paref{var_aelle} is
$
\Var(a_{\ell,m;2}) = \frac{4\pi}{2\ell+1}
$
and $G_{\ell;2} \equiv P_\ell$ the $\ell$-th Legendre polynomial \cite[\S 4.7]{szego}. 

It is readily checked that the addition formula \cite[\S 9.6]{andrews} for hyperspherical harmonics 
\begin{equation}\label{addformula}
\frac{|\mathbb S^d|}{n_{\ell;d}} \sum_{m=1}^{n_{\ell;d}}Y_{\ell,m;d}(x) Y_{\ell,m;d}(y) = G_{\ell;d}(\cos d(x,y)),\qquad x,y\in \mathbb S^d
\end{equation}
ensures that the random field $T_\ell$ as defined in \paref{rf} has covariance kernel given by \paref{cov}. 

The defect $D_\ell:=D(T_\ell)$  in \paref{def} of $T_\ell$ is then a random variable defined as
\begin{equation}\label{defect}
D_\ell = \int_{\mathbb S^d} \mathcal H(T_\ell(x))\,dx.
\end{equation}
We are interested in the asymptotic behavior of the sequence $(D_\ell)_\ell$ in the high-energy limit, i.e. as $\ell\to +\infty$. We anticipate here that $D_\ell$ vanishes for odd $\ell$, therefore we will study the defect only for even integers $\ell$ (we will prove it in \S \ref{on the var}). In particular,  for $\ell\to +\infty$ we shall mean: \emph{as $\ell\to +\infty$ along even integers}. 

\subsection{Previous work}\label{related}

The case $d=2$ has been investigated by Marinucci and Wigman. In \cite{def}, they prove that $D_\ell$ is centered and give an asymptotic result for the variance, i.e. as $\ell \to +\infty$ 
\begin{equation}\label{var2}
\Var(D_\ell) = \frac{C}{\ell^2}(1 + o(1)),\qquad C> \frac{32}{\sqrt{27}}.
\end{equation}
 In \cite{Nonlin}, a Central Limit Theorem is shown for the defect on the $2$-sphere: as $\ell\to +\infty$
\begin{equation}\label{clt2}
\frac{D_\ell}{\sqrt{\Var(D_\ell)}}\mathop{\to}^{\mathcal L} Z,
\end{equation}
where $Z\sim \mathcal N(0,1)$ is a standard Gaussian random variable.

Observe now that a simple transformation gives
\begin{equation*}
D_\ell = 2\int_{\mathbb S^d} 1_{(0,+\infty)}(T_\ell(x))\,dx - |\mathbb S^d|,
\end{equation*}
where $\int_{\mathbb S^d} 1_{(0,+\infty)}(T_\ell(x))\,dx=:S_\ell(0)$ is the measure of the so-called $0$-excursion set $\lbrace x\in \mathbb S^d : T_\ell(x) >0\rbrace$. The general case of $z$-excursion set for $z\in \mathbb R$, on the $d$-sphere ($d\ge 2)$ has been studied  in \cite{maudom}.  In the latter reference,  quantitative CLTs in the Wasserstein distance for the measure $S_\ell(z):=\int_{\mathbb S^d} 1_{(z,+\infty)}(T_\ell(x))\,dx$ of $z$-excursion sets $\lbrace x\in \mathbb S^d : T_\ell(x) >z\rbrace$ are given (see below \paref{varmaudom} and \paref{wassmaudom}), \emph{except for the nodal case $z=0$}. Recall that Wasserstein distance (e.g. \cite[\S C.2]{noupebook}) is the probability metric between two random variables $N,Z$ defined as 
\begin{equation}\label{wass}
\text{d}_W\left ( N, Z   \right) := \sup_{h\in \text{Lip}_1} |\E[h(N)]-\E[h(Z)]|,
\end{equation}
where $\text{Lip}_1$ denotes the set of Lipschitz functions whose Lipschtiz constant equals $1$. 

From \cite{maudom} for $z\ne 0$, we have that 
\begin{equation}\label{varmaudom}
\Var(S_\ell(z))\sim  |\mathbb S^d|^2(d-1)!\,\frac{z^2\phi(z)^2}{4}\times \frac{1}{\ell^{d-1}},\qquad \text{as } \ell\to +\infty,
\end{equation}
 $\phi$ (resp. $\Phi$) denoting the standard Gaussian density (resp. distribution function),
and moreover 
\begin{equation}\label{wassmaudom}
\text{d}_W\left(\frac{S_\ell(z)-|\mathbb S^d|(1-\Phi(z))}{\sqrt{\Var(S_\ell(z))}}, Z \right) = O\left( \frac{1}{\sqrt{\ell}}\right),
\end{equation}
where $Z\sim \mathcal N(0,1)$ as before. In particular, \paref{wassmaudom} implies a CLT for the measure of excursion sets at any non-zero level.  

\subsection{Main results}\label{mainsec}

In this paper we study the high-energy behavior of the sequence of random variables $(D_\ell)_\ell$ \paref{defect} in any dimension $d\ge 2$. Evaluating the mean of $D_\ell$ is trivial. Indeed, exchanging the expectation with integration over $\mathbb S^d$ we get
$$
\E[D_\ell]= \int_{\mathbb S^d} \E[\mathcal H(T_\ell(x))]\,dx,
$$
and since for every $x\in \mathbb S^d$, $\E[\mathcal H(T_\ell(x))]=0$ by the symmetry of the Gaussian distribution, we have just proved the following. 
\begin{lemma}\label{lem media}
For every $\ell \in \mathbb N$
$$
\E[D_\ell]=0.
$$
\end{lemma}
For the variance, we have the following asymptotic result which generalizes \paref{var2} to the higher dimensional sphere and whose proof is given in \S \ref{proofvar}.
\begin{proposition}\label{prop var}
As $\ell\to +\infty$, the defect variance $\Var(D_\ell)$ satisfies
\begin{equation}\label{vareq}
\Var(D_\ell) = \frac{C_d}{\ell^d}(1 +o(1)),
\end{equation}
where $C_d > 0$ is a positive constant depending only on $d$.  
\end{proposition}
Note that $C_2=C$ in \paref{var2}. Comparing  \paref{vareq} with \paref{varmaudom}, one infers that the variance of the measure of $z$-excursion sets has a smaller order of magnitude in the nodal case than for $z\ne 0$. This phenomenon appears in many situations and it is usually referred to as Berry's cancellation phenomenon \cite{AmP,Wig}. See \S\ref{det vs ran}. 

The constant $C_d$ in Proposition \ref{prop var} may be expressed in terms of the improper (conditionally convergent) integral
\begin{equation}\label{intC}
C_d=\frac{4}{\pi}|\mathbb S^d| |\mathbb S^{d-1}|\int_{0}^{+\infty}
\psi^{d-1}\left( \arcsin \left (\widetilde J_{d}(\psi)\right)              
-  \widetilde J_{d}(\psi)           \right)\,d\psi,
\end{equation}
where
\begin{equation}\label{tildej}
\widetilde J_{d}(\psi):=2^{d/2 - 1}\left (d/2-1 \right)!\,
J_{d/2-1}(\psi )\psi ^{-\left( d/2 
-1\right)},\qquad \psi>0
\end{equation}
and 
$J_{d/2-1}$ denotes the Bessel function \cite[\S 1.7]{szego} of order $d/2-1$. See \paref{guess1} for a formula equivalent to \paref{intC} that expresses $C_d$ as a convergent series. 
We do not know whether one can evaluate $C_d$
explicitly; however, we shall
show that
\begin{equation}\label{disug}
C_d>\frac{2}{3\pi}|\mathbb S^d| |\mathbb S^{d-1}| \left(2^{d/2 - 1}\left (d/2-1 \right)!\right)^3
\frac{3^{d/2 -3/2}}{2^{3\left (d/2-1 \right)-1}\sqrt \pi\,
\Gamma \left ( d/2 -1/2  \right )},
\end{equation}
$\Gamma$ denoting the Gamma function \cite[\S 1.7]{szego}. For instance, for $d=2$, \paref{disug} gives $C_2 > 32/\sqrt{27}$, that coincides with \paref{var2}.

The main contribution of the present paper is the following  quantitative CLT  in the Wasserstein distance \paref{wass} which extends and generalizes the results from \cite{maudom, Nonlin} collected in \S \ref{related}.
We cover the nodal case which is more interesting (see \S\ref{det vs ran}) than the non-zero level case treated in \cite{maudom}. Moreover, there are marked differences compared to the latter (see \S\ref{on the proof}); indeed, the techniques developed in \cite{maudom} are \emph{not} enough to deal with the defect, and we have to overcome some additional difficulties. Finally, our result is stronger than \paref{clt2} (proven in \cite{Nonlin}), yielding also the rate of convergence to the Gaussian distribution. 
\begin{theorem}\label{th qclt}
Let $Z$ be a standard Gaussian random variable. For $d\ge 2$ we have, as $\ell\to +\infty$,
$$
\text{d}_{W}\left (\frac{D_\ell}{\sqrt{\Var(D_\ell)}}, Z \right ) = O\left( \frac{1}{^4\sqrt{\log \ell}}\right),
$$
in particular
$$
\frac{D_\ell}{\sqrt{\Var(D_\ell)}}\mathop{\goto}^{\mathcal L} Z.
$$
\end{theorem}
The proof of Theorem \ref{th qclt} will be given in \S \ref{proofqclt} and requires also some intermediate key results  that for 
 $d=2$ and $d>5$ have been shown in \cite{maudom}. We are able to solve the remaining cases $d=3,4,5$ improving also previous results in \cite[Proposition 2.3]{maudom}, by means of next Lemma \ref{lem} and Lemma \ref{lemCG}.  See \S \ref{on asymp} for motivating details and further discussions. 

Let us denote by $H_3$ the third Hermite polynomial, i.e. $H_3(t) = t^3 -3t,$ $t\in \mathbb R$, and by $\text{d}_{\mathcal D}$ either the Wasserstein \paref{wass}, Kolmogorov or Total Variation distance (see \cite[\S C.2]{noupebook}), then
\begin{lemma}\label{lem}
For $d\ge 2$, as $\ell\to +\infty$ 
\begin{equation}\label{ciaone}
\text{d}_{\mathcal D}\left( \frac{\int_{\mathbb S^d} H_3(T_\ell(x))\,dx}{\sqrt{\Var\left( \int_{\mathbb S^d} H_3(T_\ell(x))\,dx  \right)}}, Z \right) = O\left ( \sqrt{\frac{\text{cum}_4\left( \int_{\mathbb S^d} H_3(T_\ell(x))\,dx   \right )}{\Var\left (\int_{\mathbb S^d} H_3(T_\ell(x))\,dx  \right )^2}}   \right ) =O\left(\frac{1}{ \sqrt{\ell^{d-1}}}   \right),
\end{equation} 
where $\text{cum}_4\left( \int_{\mathbb S^d} H_3(T_\ell(x))\,dx   \right )$ denotes the fourth cumulant\footnote{see \cite[(3.1.3)]{P-T}} of $\int_{\mathbb S^d} H_3(T_\ell(x))\,dx$. 
\end{lemma} 
(Note that the first equality in \paref{ciaone} is an application of the so-called Fourth Moment Theorem \cite[Theorem 5.2.7]{noupebook}.)
In particular,  Lemma \ref{lem} entails that, as $\ell\to +\infty$, 
$$\frac{\int_{\mathbb S^d} H_3(T_\ell(x))\,dx}{\sqrt{\Var\left( \int_{\mathbb S^d} H_3(T_\ell(x))\,dx  \right)}} \mathop{\goto}^{\mathcal L} Z,$$
where $Z$ is a standard Gaussian random variable. 
Note that, for $d\ge 3$, (2.21) in \cite[Proposition 2.3]{maudom} gives only  $O\left( \ell^{-(d-5)/4}  \right)$ for the l.h.s. of \paref{ciaone} - which does not vanish when $d\in \lbrace 3,4,5 \rbrace$, as $\ell\to +\infty$. 

The random variable $\int_{\mathbb S^d} H_3(T_\ell(x))\,dx$ is the so-called bispectrum of $T_\ell$ and it is of independent interest (see \cite{marinucci2006,marinucci2008}). 
In particular for $d=2$, the information on the bispectrum of $T_\ell$ are used to test some features of Cosmic Microwave Background \cite{dogiocam}; on the $2$-sphere, \paref{ciaone} coincides with the result found by Marinucci in \cite{marinucci2008}.

Our argument in order to prove Lemma \ref{lem} 
requires technical computations involving concatenated sums of integrals of three hyperspherical harmonics of the form 
\begin{equation}\label{defcg}
{\mathcal G}^{\ell,m_3}_{\ell,m_1,\ell,m_2;d}:=\int_{\mathbb S^d} Y_{\ell,m_1;d}(x) Y_{\ell,m_2;d}(x) Y_{\ell,m_3;d}(x)\,dx,
\end{equation}
for $m_1,m_2,m_3\in \lbrace 1,2,\dots , n_{\ell;d} \rbrace$; (note that by definition, ${\mathcal G}^{\ell,m_3}_{\ell,m_1,\ell,m_2;d}$ is invariant under any permutations of indexes $m_1,m_2,m_3$).  The integral
${\mathcal G}^{\ell,m_3}_{\ell,m_1,\ell,m_2;d}$ is strictly related to so-called Clebsch-Gordan coefficients \cite{faraut, dogiocam, vilenkin2} for the special orthogonal group $SO(d+1)$, which  play a key role in group representation properties of the latter. 

The integral in \paref{defcg} is well-known  for $d=2$ (so-called Gaunt formula - see \cite[Proposition 3.43]{dogiocam})
and several applications by many authors can be found (see \cite{vale3, marinucci2006, marinucci2008, dogiocam, Nonlin} e.g.), because of its importance also in the quantum theory of angular momentum. From \cite[(5.6.2.12),(5.6.2.13)]{varshalovich}, for the two-dimensional case,
\begin{equation}\label{gaunt}
{\mathcal G}^{\ell,m_3}_{\ell,m_1,\ell,m_2;2}= \sqrt{\frac{2\ell+1}{4\pi}} C_{\ell,m_1,\ell,m_2}^{\ell,m_3}\cdot C^{\ell,0}_{\ell,0,\ell,0},
\end{equation}
where $C_{\ell_1,m_1,\ell_2,m_2}^{\ell_3,m_3}$ denote Clebsch-Gordan coefficients (see \cite[Chapter 8]{varshalovich} or \cite[\S 3.5]{dogiocam}) for the group $SO(3)$; explicit formulas are known for the latter. (Note that one usually considers $m\in \lbrace -\ell, \dots, \ell \rbrace$ instead of $m\in \lbrace 1,2,\dots, 2\ell+1\rbrace$, see e.g. \cite[\S 3]{dogiocam}.)
 To the best of our knowledge, analogous estimates as those for the $2$-dimensional case are not available in the literature for higher dimensional spheres; in what follows, we therefore need to develop some \emph{novel} tools in order to complete our argument. 

 The main achievement in this direction is the following result whose proof, given in \S \ref{CG}, relies on simple ideas that however could be used in other circumstances involving Clebsch-Gordan coefficients for any topological compact group. Moreover, this kind of results can be applied to solve problems on the hypersphere as those investigated in \cite{marinucci2006,marinucci2008} for the two-dimensional case; we believe they could be useful in order to study also other statistical issues on $\mathbb S^d$, a topic which has recently received some attention (see e.g. \cite{claudiohyper}).

\begin{lemma}\label{lemCG}
For every even $\ell \in \mathbb N$, $M, M'\in \lbrace 1,2,\dots, n_{\ell;d}\rbrace$ and $d\ge 2$
$$
\sum_{m_1,m_2=1}^{n_{\ell;d}} {\mathcal G}^{\ell,M}_{\ell,m_1,\ell,m_2;d}\,{\mathcal G}^{\ell,M'}_{\ell,m_1,\ell,m_2;d}= \delta_{M}^{M'} \frac{(n_{\ell;d})^2}{|\mathbb S^d|}\frac{|\mathbb S^{d-1}| }{|\mathbb S^{d}|}\int_{-1}^1 G_{\ell;d}(t)^3 \left(\sqrt{1-t^2}\right)^{d-2}dt. 
$$
\end{lemma}
In the case $d=2$, from \paref{gaunt}, by orthonormality properties of Clebsch-Gordan coefficients \cite[(3.62)]{dogiocam} and since $(2\ell+1)^{-1}\left( C^{\ell,0}_{\ell,0,\ell,0} \right)^2= \int_{-1}^1 P_{\ell}(t)^3\,dt$ (again from \paref{gaunt}, see also \cite{Nonlin}), we have
\begin{equation*}
\sum_{m_1,m_2} {\mathcal G}^{\ell,M}_{\ell,m_1,\ell,m_2;2}\,{\mathcal G}^{\ell,M'}_{\ell,m_1,\ell,m_2;2} = \delta_{M}^{M'} \frac{(2\ell+1)^2}{2\cdot 4\pi}\int_{-1}^1 P_{\ell}(t)^3\,dt,
\end{equation*}
which coincides with the statement of Lemma \ref{lemCG}.

\subsection{Plan of the paper}\label{plan}

In \S \ref{det vs ran} we give motivations for our work. 
In \S \ref{on the proof} we briefly explain our argument to prove the main results in \S\ref{mainsec}, whereas \S\ref{wienerintro} fixes some notation and background about Wiener chaoses and Stein-Malliavin techniques for distributional approximations. The proof of Proposition \ref{prop var} is given in \S\ref{proofvar} and we prove Theorem \ref{th qclt} in \S \ref{proofqclt}, while \S\ref{345} deals with Lemma \ref{lem} and Lemma \ref{lemCG}. Finally, we collect some technical computations and intermediate results in the Appendix \S\ref{app}.

\subsection{Acknowledgements}

This topic was suggested by Domenico Marinucci. 
The author would like to thank him, Giovanni Peccati and Igor Wigman for useful conversations, and an anonymous referee for insightful comments. 

The research leading to this work was carried out within the framework of the ERC \emph{Pascal} project no. 277742 and of the Grant STARS (R-AGR-0502-10) at Luxembourg University. The author is currently supported by the Fondation Sciences Math\'ematiques de Paris and the ANR-17-CE40-0008 project \emph{Unirandom}. 

\section{Outline of the paper}

\subsection{Motivations}\label{det vs ran}

Berry argued that, at least for ``generic" chaotic surfaces, the local behavior of (deterministic) eigenfunctions should be \emph{universal} \cite{berry}. He proposed to compare the eigenfunction $f$ in \paref{eqSch} of large eigenvalue $E$ to a ``typical" instance of a monochromatic random wave with wavenumber $\sqrt{E}$; we can define the latter as the centered Gaussian field $W=(W(x))_{x\in \mathbb R^2}$ on the plane, whose covariance structure is given by 
\begin{equation}\label{berrycov}
\Cov(W(x), W(y)) = J_0\left (\sqrt{E}|x-y| \right ),\qquad x,y\in \mathbb R^2,
\end{equation}
$J_0$ being the $0$-order Bessel function \cite[\S 1.7]{szego}. Local properties of $f$ can then be predicted by $W$; for instance, nodal lines of the latter should model nodal lines of $f$ (see \cite{wigsurvey}). 

In the spherical case, from \paref{cov} with $d=2$, the random model \paref{rf} has covariance kernel given by 
$
\Cov(T_\ell(x),T_\ell(y))=P_\ell(\cos d(x,y))
$, $x,y\in \mathbb S^2$, where $P_\ell \equiv G_{\ell;2}$ is still the $\ell$-th Legendre polynomial \cite[\S 4.7]{szego}. Hilb's asymptotics \cite[Theorem 8.21.6]{szego} gives, for large eigenvalues, 
$$
P_\ell(\cos \theta)\sim \sqrt{\frac{\theta}{\sin \theta}}\, J_0\left ((\ell +1/2)\theta \right )
$$
uniformly for $\theta\in [0, \pi - \varepsilon]$, 
similar to \paref{berrycov} but for the square root that keeps a trace about the geometry of the sphere. In recent years, the geometry of random spherical eigenfunctions has been studied by several papers, motivated also by applications in Cosmology - for instance concerning the analysis of CMB \cite{dogiocam}. In particular, so-called Lipschitz Killing curvatures \cite[\S 6.3]{adlertaylor} have been investigated; namely (in dimension $d=2$), the boundary length \cite{mauphd, Wig, wigsurvey, MRW17}, the area \cite{maudom, def, Nonlin} and the Euler-Poincar\'e characteristic \cite{eulvale, fluct} of excursion sets at any level. For each of the just mentioned geometric functionals, the same qualitative behavior has been observed, i.e. a lower-order asymptotic variance in the nodal case (see \cite[\S 1.2]{eulvale} for an overview). 

This paper plays a key role in the analysis of Lipschitz-Killing curvatures on the hypersphere; indeed, we complete the investigation of the empirical volume of excursion sets that was started in \cite{maudom}. 
Proposition \ref{prop var} compared to \paref{varmaudom} confirms the behavior predicted by Berry for the variance \cite{berry2002} also in this case: it is clear that the leading constant in \paref{varmaudom} vanishes for $z=0$. This phenomenon has a deeper interpretation related to chaotic expansions which also explain the rate we obtain in the nodal case (Theorem \ref{th qclt} compared to \paref{wassmaudom}); we shall be back on this issue in \S\ref{on rates}. 

A possible future research topic can be the investigation of all the others Lipschitz-Killing curvatures on the hypersphere in any dimension or, more generally, any \emph{nice} compact manifold (as the multidimensional torus $\mathbb T^d:=\mathbb R^d/\mathbb Z^d$).

\subsection{On the proofs of the main results}\label{on the proof}

Our approach to establish Proposition \ref{prop var} and Theorem \ref{th qclt} relies, among other things, on   the chaotic decomposition of $D_\ell$ \paref{defect} (this technique was used also in \cite{Nonlin, misto, MRW17} e.g.). Since the defect is a square integrable fuctional of a Gaussian field, it can be written as a series
$
D_\ell = \sum_{q=0}^{+\infty} D_\ell[q],
$
converging in $L^2(\P)$, where $D_\ell[q]$ is the orthogonal projection of $D_\ell$ onto the so-called $q$-th Wiener chaos ($D_\ell[0]=\E[D_\ell]$); the random variables  $D_\ell[q]$, $q\ge 0$ are pairwise orthogonal. 
More precisely, we have the following statement, whose proof is given in \S\ref{appchaos}. Recall that Hermite polynomials \cite[\S 5.5]{szego} $(H_k)_{k\ge 0}$ are defined as $H_0(t)=1$ and for $k\ge 1$
$$
H_k(t) := (-1)^k \phi(t)^{-1} \frac{d^k}{dt^k} \phi(t),\qquad t\in \mathbb R,
$$
where $\phi$ still denotes the probability density of a standard Gaussian random variable. For instance, $H_1(t)=t$, $H_2(t)=t^2-1$, $H_3(t) = t^3 - 3t$ and so on.
\begin{lemma}\label{lem chaos}
The Wiener-It\^o chaos decomposition of the defect $D_\ell$ is 
\begin{equation}\label{chaosexpintro}
D_\ell =  \sum_{q=1}^{+\infty} D_\ell[2q+1]=\sum_{q=1}^{+\infty} \frac{J_{2q+1}}{(2q+1)!}\int_{\mathbb S^d}H_{2q+1}(T_\ell(x))\,dx,
\end{equation}
where $$
J_{2q+1} := \frac{2}{\sqrt{\pi}} H_{2q}(0).
$$
\end{lemma}
The first marked difference compared to the non-zero level case treated in \cite{maudom} is that the non-linear functionals investigated in the latter reference are assumed to have \emph{non}-vanishing (and asymptotically leading) second chaotic component, whereas from \paref{chaosexpintro} we have $D_\ell[2]=0$ for every $\ell$. Some of the tools used in \cite{maudom} are hence not suitable, and we will need some new technical results. We will clarify this point in what follows. 

\subsubsection{On the variance}\label{on the var}

Our argument  to prove Proposition \ref{prop var} is essentially equivalent to the one obtained by generalizing to higher dimensional setting the approach used in \cite{def} to find the defect asymptotic variance on the $2$-sphere.

Let us keep in mind \paref{chaosexpintro}. By virtue of the orthogonality and isometric properties of chaotic projections and some straightforward computations, we have that 
\begin{equation}\label{var introtilde}
\Var(D_\ell) = |\mathbb S^d||\mathbb S^{d-1}|\sum_{q=1}^{+\infty} \frac{J_{2q+1}^2}{(2q+1)!} \int_{0}^{\pi} G_{\ell;d}(\cos \vartheta)^{2q+1}(\sin \vartheta)^{d-1}\,d\vartheta;
\end{equation}
since Gegenbauer polynomials are symmetric, that is, $G_{\ell;d}(-t)=(-1)^\ell G_{\ell;d}(t)$, $t\in [-1,1]$ (see e.g. \cite[\S 4.7]{szego}), we deduce that the integral in the r.h.s. of \paref{var introtilde} vanishes for odd $\ell$. For even $\ell$, we have
\begin{equation}\label{var intro}
\Var(D_\ell)=2\sum_{q=1}^{+\infty} \frac{J_{2q+1}^2}{(2q+1)!} |\mathbb S^d||\mathbb S^{d-1}|\int_{0}^{\pi/2} G_{\ell;d}(\cos \vartheta)^{2q+1}(\sin \vartheta)^{d-1}\,d\vartheta.
\end{equation}
Lemma \ref{lem media} hence implies that 
 $D_\ell=0$ for odd $\ell$, as anticipated in \S\ref{introH}. 
\cite[Proposition 1.1]{maudom} gives, as $\ell\to +\infty$,
$$
\int_{0}^{\pi/2} G_{\ell;d}(\cos \vartheta)^{2q+1}(\sin \vartheta)^{d-1}\,d\vartheta = \frac{c_{2q+1;d}}{\ell^d}(1+o(1)),
$$
for some nonnegative constant $c_{2q+1;d}$ (in \cite{maudom,def} it was conjectured that $c_{2q+1;d}>0$ for every $q$ and $d$). One would hence expect Proposition \ref{prop var} to hold with 
$$C_d = 2|\mathbb S^d||\mathbb S^{d-1}|\sum_{q=1}^{+\infty} \frac{J_{2q+1}^2}{(2q+1)!}c_{2q+1;d} .$$
In \S\ref{proofvar} we will properly prove it. 
Recall now that the second moment of Gegenbauer polynomials is 
$$
\int_{0}^{\pi} G_{\ell;d}(\cos \vartheta)^{2}(\sin \vartheta)^{d-1}\,d\vartheta = \frac{|\mathbb S^d|}{|\mathbb S^{d-1}|\, n_{\ell;d}},
$$
in particular it is of bigger order, as $\ell\to +\infty$, than the $q$-th moment for any $q\ge 3$. This make easier the investigation of the variance of non-linear functionals whose second chaotic component is non-zero and asymptotically leading (as those considered in \cite{maudom}). Indeed, in the latter case the leading term for the asymptotic variance corresponds to the second chaotic projection, whereas for the defect all terms equally contribute. 

\subsubsection{On the asymptotic distribution}\label{on asymp}

As a consequence of what was mentioned just above in \S\ref{on the var}, 
the limiting distribution of functionals whose second chaotic component is non-zero equals to the asymptotic law of only \emph{one} chaotic term (the second). Since the latter is proportional to a sum of independent random variables, an application of the standard Central Limit Theorem allows to conclude the investigation. For the defect instead, we have to study the asymptotic behavior of every chaotic component.

The CLT recalled in \paref{clt2} and \paref{wassmaudom} make it plausible  to conjecture that  the defect is asymptotically Gaussian in any dimension; we wish hence to extend and generalize \paref{clt2} and \paref{wassmaudom}.
To prove a CLT for the series \paref{chaosexpintro},
we first need to deal with single chaotic components.  Theorem 1.2 in \cite{maudom} ensures that $D_\ell[2q+1]$ is, as $\ell\to +\infty$, Gaussian for every pair $(2q+1,d)$ except for $(3,d)$  when $d=3,4,5$. It is however  reasonable to believe that also $D_\ell[3]$ is asymptotically normal in any dimension, as Lemma \ref{lem} states. 

To prove the latter, since we are in a fixed chaos - the third one - we can use Fourth Moment Theorem \cite[Theorem 5.2.7]{noupebook} i.e., to have asymptotic Gaussianity it is enough (and necessary) that
\begin{equation}\label{4th}
\frac{\text{cum}_4\left(\int_{\mathbb S^d}H_3(T_\ell(x))\,dx   \right)}{\left(\Var\left (\int_{\mathbb S^d}H_3(T_\ell(x))\,dx \right )\right)^2}\to 0,
\end{equation}
where by $\text{cum}_4(X)$ we denote the $4$-th cumulant \cite[(3.1.3)]{P-T} of the random variable $X$. Note that from Thorem 5.2.6 in \cite{noupebook}, the l.h.s. of \paref{4th} allows to estimate moreover the rate of convergence to the Gaussian distribution in various probability metrics \cite[\S C.2]{noupebook}, Wasserstein distance \paref{wass} included (see \S\ref{4th mom}).

By the properties of cumulants \cite[\S 3.1]{P-T}, we have 
\begin{equation*}
\begin{split}
&\text{cum}_4 \left(\int_{\mathbb S^d}H_3(T_\ell(x))\,dx\right)\cr
&= \int_{(\mathbb S^d)^4} \text{cum}\left( H_3(T_\ell(w)),H_3(T_\ell(z)),H_3(T_\ell(w')),H_3(T_\ell(z'))\right)\,dw dz dw' dz'. 
\end{split}
\end{equation*}
The diagram formula for Hermite polynomials \cite[Proposition 4.15]{dogiocam} applied to the integrand of the r.h.s. of the last equality and results by Nourdin and Peccati, in particular \cite[Lemma 5.2.4]{noupebook}, ensure that the major contribution for the fourth cumulant of $\int_{\mathbb S^d}H_3(T_\ell(x))\,dx$ comes from so-called \emph{circulant diagrams}. Indeed, in \cite[Lemma 4.1]{maudom} it has been shown that the contribution of the latter can be expressed as multiple integrals of products of powers of Gegenbauer polynomials. We are then left in our case with the following 
\begin{equation}\label{circdiag}
\begin{split}
I_{\ell;d}:=\int_{(\mathbb S^d)^4} &G_{\ell;d}(\cos d(w,z))G_{\ell;d}(\cos d(w,w'))^2 \cr
& \times G_{\ell;d}(\cos d(w',z'))G_{\ell;d}(\cos d(z',z))^2\,dw dz dw' dz'.
\end{split}
\end{equation}
In \cite[Proposition 4.2, Proposition 4.3]{maudom} to prove \paref{4th}, upper bounds for \paref{circdiag} are given which are too big for $d=3,4,5$: the technique consists first of bounding the contribution of one of the involved Gegenbauer polynomials by its sup-norm ($=1)$ and then of using Cauchy-Schwarz inequalities. This reduces to deal simply with moments of Gegenbauer polynomials but allows to get only not satisfactory bounds.  
Our argument is \emph{subtler} and more difficult, since we have to compute the \emph{exact} asymptotics for $I_{\ell;d}$ in \paref{circdiag}, as $\ell\to +\infty$. 

The addition formula \paref{addformula} applied several times in \paref{circdiag}, leads to concatenated sums of integrals of three hyperspherical harmonics. With the same notation as in \paref{defcg} we find
\begin{equation*}
I_{\ell;d}=\left(\frac{|\mathbb S^d|}{n_{\ell;d}}\right)^6\sum_{m_1,m_2,m_3,m_1',m_2',m_3'=1}^{n_{\ell;d}}{\mathcal G}^{\ell,m_1}_{\ell,m_2',\ell,m_3';d}{\mathcal G}^{\ell,m_1'}_{\ell,m_2,\ell,m_3;d}{\mathcal G}^{\ell,m_1}_{\ell,m_2,\ell,m_3;d}{\mathcal G}^{\ell,m_1'}_{\ell,m_2',\ell,m_3';d}.
\end{equation*}
The idea now is to find some useful property for double sums of these coefficients \paref{defcg}  (as done in Lemma \ref{lemCG}). 
Once Lemma \ref{lem} is proved, then an argument similar to the one given in the proof of Corollary 4.2 in \cite{Nonlin}, provides a CLT for the defect in any dimension. We are however interested in subtler results: rates of convergence to the limiting distribution - as explained in the next section. 

\subsubsection{On rates of convergence}\label{on rates}

In order to obtain rates of convergence, i.e. to prove our Theorem \ref{th qclt}, it is \emph{not} enough to generalize to higher dimensional setting the approach used in \cite{Nonlin}. Indeed in particular one needs  refined estimates on the distance in distribution between the ``truncated" defect (see just below) and the Gaussian law. 

Let us truncate the series \paref{chaosexpintro} at some frequency $m$, obtaining $D_\ell^m=\sum_{q=1}^{m} D_\ell[2q+1]$. We know that $D_\ell^m$ is Gaussian, as $\ell\to +\infty$, (since it is a linear combination of asymptotically normal random variables living in different order chaoses; see \cite{peccatitudor}) and we can then compute the rate of convergence in Wasserstein distance \paref{wass} to the limiting distribution by using Stein-Malliavin techniques for Normal approximations \cite[Chapters 5, 6]{noupebook}.  
The tail $D_\ell - D_\ell^m$ can be controlled by its $L^2(\P)$-norm. Summing up all contributions  and choosing an optimal speed  $m=m(\ell)$, we obtain Theorem \ref{th qclt}.

The rate of convergence is slower than in the non-nodal case \paref{wassmaudom}. Indeed, here all chaoses in the Wiener-It\^o expansion for the defect \paref{chaosexpintro} contribute, whereas for $z\ne 0$ the second chaos does not vanish and hence gives the dominating term, making easier also the estimation of the speed of convergence. For the defect instead, one has to control the whole series. 

As remarked also in \cite[\S1.2]{eulvale}, for Lipschitz-Killing curvatures \cite[\S 6.3]{adlertaylor} on the sphere (and on other manifolds, such as the torus - see \cite{misto}) the second chaos dominates only in the non-nodal case, giving a powerful explanation for Berry's cancellation phenomenon concerning the variance. 

\section{Wiener chaoses and Stein-Malliavin techniques}\label{wienerintro}

For a complete discussion on the following topics see  \cite{noupebook}. 

\subsection{Random eigenfunctions as isonormal Gaussian processes}

Let $X=(X(f))_{f\in L^2(\mathbb S^d)}$ be an isonormal Gaussian process on $L^2(\mathbb S^d)$, the space of square integrable functions on the $d$-sphere ($d\ge 2$).
We mean a real-valued centered Gaussian field  indexed by the elements of $L^2(\mathbb S^d)$ verifying the isometric property, i.e. 
\begin{equation*}
\Cov(X(f),X(h) ) = \langle f, h \rangle_{L^2(\mathbb S^d)}:=\int_{\mathbb S^d}f(x)h(x)\,dx,\qquad f,h\in L^2(\mathbb S^d). 
\end{equation*}
We can construct it as follows. Let us denote by $\mathcal B(\mathbb S^d)$ the Borel $\sigma$-field on $\mathbb S^d$ and let $W=\lbrace W(A): A\in \mathcal B(\mathbb S^d)\rbrace$ be a centered Gaussian family on $\mathbb S^d$  such that
$$
\Cov(W(A),W(B)) = \int_{\mathbb S^d} 1_{A\cap B}(x)\,dx,
$$
where $1_{A\cap B}$ denotes the indicator function of the set $A\cap B$. 
The random field $X$ on $L^2(\mathbb S^d)$ defined as
\begin{equation}\label{iso}
X(f) := \int_{\mathbb S^d} f(x)\,dW(x),\qquad f\in L^2(\mathbb S^d),
\end{equation}
i.e. the Wiener-It\^o integral of $f$ with respect to the Gaussian measure $W$, is the isonormal Gaussian process on the $d$-sphere. 

Note that for random eigenfunctions \paref{rf} it holds
\begin{equation}\label{eigen}
T_\ell(x) \mathop{=}^{\mathcal L} X(f_{\ell,x}),\qquad x\in \mathbb S^d,
\end{equation}
as stochastic processes, 
where $X$ is defined as in \paref{iso} and $f_{\ell,x}:= f_{\ell,x;d}$ is given by
\begin{equation}\label{kernel}
f_{\ell,x}(y):= \sqrt{\frac{n_{\ell;d}}{|\mathbb S^d|}} G_{\ell;d}(\cos d(x,y)),\qquad y\in \mathbb S^d.
\end{equation}

\subsection{Defect in Wiener chaoses}\label{defchaos}

Let us now recall the notion of Wiener chaos, mentioned in \S\ref{on the proof}. As before,  let $(H_k)_{k\ge 0}$ denote the sequence of Hermite polynomials \cite[\S 5.5]{szego}.  Normalized Hermite polynomials $\left(   \frac{H_k}{\sqrt{k!}} \right)_{k\ge 0}$ form an orthonormal basis of $L^2(\R, \phi(t)\,dt)$, the space of square integrable functions on the real line endowed with the Gaussian measure. 
Recall that for jointly Gaussian random variables $Z_1,Z_2\sim \mathcal N(0,1)$ and $k_1,k_2\ge 0$, we have
\begin{equation}\label{producthermite}
\E[H_{k_1}(Z_1) H_{k_2}(Z_2)] = k_1!\left( \E[Z_1 Z_2] \right)^{k_1} \delta_{k_1}^{k_2}.
\end{equation}
For each integer $q\ge 0$, consider the closure $C_q$ in $L^2(\P)$ of the linear space generated by random variables of the form
\begin{equation*}\label{hermrv}
H_q(X(f)),\qquad f\in L^2(\mathbb S^d),\  \| f\|_{L^2(\mathbb S^d)} =1. 
\end{equation*}
The space $C_q$ is the so-called $q$-th Wiener chaos associated with $X$. By \paref{producthermite}, it is easy to check that $C_q \perp C_{q'}$ for $q\ne q'$ and moreover the following Wiener-It\^o chaos decomposition holds
$$
L^2(\P) = \bigoplus_{q=0}^{+\infty} C_q,
$$
i.e. every random variable $F$ whose second moment is finite can be expressed as a series converging in $L^2(\P)$
\begin{equation}\label{chaos dec}
F=\sum_{q=0}^{+\infty} F[q],
\end{equation}
where $F[q] = \text{proj}(F|C_q)$ is the orthogonal projection of $F$ onto the $q$-th chaos ($F[0] = \E[F]$). 

The defect $D_\ell$, as defined in \paref{defect}, is a square integrable functional of the isonormal Gaussian process $X$ in \paref{iso}, actually it is easy to check that $|D_\ell|\le |\mathbb S^d|$ a.s. It hence admits a Wiener-It\^o chaos decomposition of the form \paref{chaos dec} (see Lemma \ref{lem chaos}).

\subsubsection{Multiple stochastic integrals}\label{multiple}

The $q$-th tensor power $L^2(\mathbb S^d)^{\otimes q}$ (resp. $q$-th symmetric power $L^2(\mathbb S^d)^{\odot q}$) of $L^2(\mathbb S^d)$ is simply $L^2((\mathbb S^d)^q, dx^q)$ (resp. $L^2_s((\mathbb S^d)^q, dx^q)$, i.e. the space of a.e. symmetric functions on $(\mathbb S^d)^q$). Let us define the (linear) operator $I_q$ for unit norm $f\in L^2(\mathbb S^d)$ as 
$$
I_q(f^{\otimes q}) := H_q(X(f))
$$
and extend it to an isometry between $L^2_s((\mathbb S^d)^q):=L^2_s((\mathbb S^d)^q, dx^q)$ equipped with the modified norm $\frac{1}{\sqrt{q!}}\| \cdot \|_{L^2_s((\mathbb S^d)^q)}$ and the $q$-th Wiener chaos $C_q$ endowed with the $L^2(\P)$-norm. 

It is well-known \cite[\S 2.7]{noupebook} that for $h\in L^2_s((\mathbb S^d)^q)$,  it holds
$$
I_q(h) = \int_{(\mathbb S^d)^q} h(x_1,x_2,\dots,x_q)\,dW(x_1) dW(x_2) \dots dW(x_q),
$$
the multiple Wiener-It\^o integral of $h$ with respect to the Gaussian measure $W$, where the domains of integration implicitly avoids diagonals. 
Therefore, $F[q]$ in \paref{chaos dec} is a stochastic multiple integral $F[q]=I_q(f_q)$, for a unique kernel $f_q\in L^2_s((\mathbb S^d)^q)$. 

From Lemma \ref{lem chaos} and \paref{eigen}, \paref{kernel}, it is immediate to check that for chaotic projections of the defect we have ($q\ge 1$)
$$
D_\ell[2q+1]= I_{2q+1}\left( \frac{J_{2q+1}}{(2q+1)!}  \int_{\mathbb S^d} f_{\ell;x}\,dx  \right).
$$

\subsubsection{Contractions}

For every $p,q\ge 1$, $f\in L^2(\mathbb S^d)^{\otimes p}, g\in L^2(\mathbb S^d)^{\otimes q}$ and $%
r=1,2,\dots, p\wedge q$, the so-called \emph{contraction} of $f$ and $g$ of
order $r$ is the element $f\otimes _{r}g\in L^2(\mathbb S^d)^{\otimes p+q-2r}$ defined as (see \cite[(B.4.7)]{noupebook})
\begin{equation}\label{contrazione}
\begin{split}
& (f\otimes _{r}g)(x_{1},\dots,x_{p+q-2r})\\
& :=\int_{(\mathbb S^d)^{r}}f(x_{1},\dots,x_{p-r},y_{1},\dots,y_{r})
g(x_{p-r+1},\dots,x_{p+q-2r},y_{1},\dots,y_{r})\,dy_{1}\dots dy_{r}.
\end{split}
\end{equation}
For $p=q=r$, we have $f\otimes _{r}g = \langle f,g \rangle_{L^2(\mathbb S^d)^{\otimes_r}}$
and for $r=0$, $f\otimes _{0}g := f\otimes g$. Denote by $f\widetilde \otimes
_{r}g$ the canonical symmetrization \cite[(B.2.1)]{noupebook} of $f\otimes_r g$. The following
multiplication formula is well-known \cite[Theorem 2.7.10]{noupebook}: for $p,q=1,2,\dots$, $f\in L^2(\mathbb S^d)^{\odot
p}, g\in L^2(\mathbb S^d)^{\odot q}$, we have
\begin{equation*}
I_p(f)I_q(g)=\sum_{r=0}^{p\wedge q} r! {\binom{p }{r}} {\binom{q }{r}}%
I_{p+q-2r}(f\widetilde \otimes _{r}g).
\end{equation*}

\subsection{Some facts about Malliavin calculus}\label{some malliavin}

In what follows, we will use standard notions and results from Malliavin calculus; we refer the reader to \cite[\S 2.3, \S 2.4]{noupebook} (from which we borrow our notation) for definitions and details.
For $q,r\ge 1$, recall that the $r$-th Malliavin derivative of a random
variable $I_{q}(f)\in C_{q}$ where $f\in L^2(\mathbb S^d)^{\odot q}$, can be identified as
the element $D^{r}I_{q}(f):\Omega \rightarrow L^2(\mathbb S^d)^{\odot r}$ given by
\begin{equation} \label{marra2}
D^{r}I_{q}(f)=\frac{q!}{(q-r)!}I_{q-r}(f),
\end{equation}%
for $r\leq q$, and $D^{r}I_{q}(f)=0$ for $r>q$. 
For notational simplicity, we shall write $D$ instead of $D^{1}$. We say
that $F$ as in (\ref{chaos dec}) belongs to the space $\mathbb{D}^{r,q}$ if ${\mathbb{\ E}}\left [|F|^{q} \right]+\dots +{
\mathbb{\ E}}\left [\Vert D^{r}F\Vert _{L^2(\mathbb S^d)^{\odot r}}^{q} \right]<+\infty$. We write
\begin{equation*}
\Vert F\Vert _{\mathbb{D}^{r,q}}:=\left( {\mathbb{\ E}}\left [|F|^{q} \right]+\dots {%
\mathbb{\ E}}\left [\Vert D^{r}F\Vert _{L^2(\mathbb S^d)^{\odot r}}^{q} \right]\right) ^{\frac{1}{q}%
}.
\end{equation*}%
It is easy to check that $F\in \mathbb{D}^{1,2}$ if and only if
\begin{equation*}
\sum_{q=1}^{\infty }q\Vert
F[q]\Vert _{L^{2}(\P )}^{2}<+\infty,
\end{equation*}
and in this case ${\mathbb{\ E}}\left [\Vert DF\Vert _{L^2(\mathbb S^d)}^{2}\right ]=\sum_{q=1}^{\infty }q\Vert
F[q]\Vert _{L^{2}(\P )}^{2}$. In particular, if $F\in L^2(\P)$ admits a finite chaotic decomposition, then it belongs to $ \mathbb{D}^{1,2}$. 
We need to introduce also the generator of the Ornstein-Uhlenbeck semigroup,
defined as
\begin{equation*}
L=-\sum_{q=0}^{\infty }q \cdot \text{proj}(\, \cdot \,| C_q),
\end{equation*}%
where $ \text{proj}(\, \cdot \,| C_q)$ is the orthogonal projection operator onto the $q$-th Wiener chaos $C_{q}$. The domain of $L$ is $\mathbb{D}^{2,2}$, equivalently the
space of Gaussian subordinated random variables $F$ such that
\begin{equation*}
\sum_{q=1}^{+\infty }q^{2}\Vert F[q]\Vert _{L^{2}(\P )}^{2}<+\infty.
\end{equation*}%
The pseudo-inverse operator of $L$ is defined as
$
L^{-1}=-\sum_{q=1}^{\infty }\frac{1}{q}\cdot \text{proj}(\, \cdot \,| C_q)
$
and satisfies, for each $F\in L^{2}(\P )$,
$
LL^{-1}F=F-{\mathbb{\ E}}[F].
$

\subsection{Fourth Moment Theorems}\label{4th mom}

\noindent Stein's method for Normal approximations and Malliavin calculus applied on Wiener chaoses lead to so-called Fourth Moment Theorems \cite{nou-pe2}, \cite[Chapters 5, 6]{noupebook}. Briefly, for a sequence of centered random variables $(F_n)_{n\ge 1}$ living in a fixed Wiener chaos $C_q$, such that $\Var(F_n)=1$ for all $n=1,2,\dots$, convergence in law to a standard Gaussian random variable is \emph{equivalent} to the sequence of fourth cumulants $(\text{cum}_4(F_n))_n$ converging to $0$ \cite[Theorem 5.2.7]{noupebook}. Moreover, the quantity $|\text{cum}_4(F_n) |$ gives information \cite[Theorem 5.2.6]{noupebook} on the rate of convergence to the limiting distribution in various probability metrics, the Wassertein distance \paref{wass} included. 
The contribution of $\text{cum}_4(F_n)$ can be expressed in terms of contractions \paref{contrazione} of the kernel $f_{q,n}$, where $F_n = I_q(f_{q,n})$ (see \cite[Lemma 5.2.4]{noupebook}). For a sequence of random vectors whose components lie in different chaoses, convergence to a multivariate Gaussian is equivalent to componentwise convergence to the normal distribution \cite{peccatitudor}. 

We will properly state these results for $H=L^2(\mathbb S^d)$ but they hold in much more generality  \cite[Chapters 5, 6]{noupebook}. Here and in what follows, $\text{d}_{TV}$ and $\text{d}_{K}$ shall denote the Total Variation and Kolmogorov distance, respectively (see \cite[\S C.2]{noupebook} e.g.). 
\begin{proposition}[Theorem 5.1.3 in \cite{noupebook}]
\label{BIGnourdinpeccati} 
Let $F\in
\mathbb{D}^{1,2}$ be such that $\mathbb{E}[F]=0,$ $\mathbb{E}[F^{2}]=\sigma 
^{2}<+\infty$ and $Z\sim \mathcal{N}(0,\sigma^2)$ a centered Gaussian random variable with variance $\sigma^2$. Then we have
\begin{equation*}
\text{d}_{W}(F,Z)\leq \sqrt{\frac{2}{\sigma ^{2}\,\pi }}\mathbb{E}%
\left [\left\vert \sigma ^{2}-\langle DF,-DL^{-1}F\rangle _{H}\right\vert \right ].
\end{equation*}%
Also, assuming in addition that $F$ has a density%
\begin{equation*}
\begin{split}
\text{d}_{TV}(F,Z) &\le \frac{2}{\sigma ^{2}}\mathbb{E}\left [\left\vert
\sigma ^{2}-\langle DF,-DL^{-1}F\rangle _{H}\right\vert \right], \\
\text{d}_{K}(F,Z) &\le \frac{1}{\sigma ^{2}}\mathbb{E}\left [\left\vert
\sigma ^{2}-\langle DF,-DL^{-1}F\rangle _{H}\right\vert \right].
\end{split}
\end{equation*}%
\end{proposition}
In the special case where $F=I_{q}(f)$ for $f\in L^2(\mathbb S^d)^{\odot q}$,
then from \cite[Theorem 5.2.6]{noupebook}, 
\begin{equation} \label{casoparticolare}
\mathbb{E}\left [\left\vert \sigma ^{2}-\langle DF,-DL^{-1}F\rangle
_{H}\right\vert \right]\leq \sqrt{\frac{1}{q^{2}}\sum_{r=1}^{q-1}r^{2}r!^{2}{%
\binom{q}{r}}^{4}(2q-2r)!\Vert f\widetilde{\otimes }_{r}f\Vert _{H^{\otimes
2q-2r}}^{2}} . 
\end{equation}
Note that in (\ref{casoparticolare}) we can replace $\Vert f\widetilde{%
\otimes }_{r}f\Vert _{H^{\otimes 2q-2r}}^{2}$ with the norm of the unsymmetryzed
contraction $\Vert f\otimes _{r}f\Vert _{H^{\otimes 2q-2r}}^{2}$, since $$\Vert f\widetilde{\otimes }_{r}f\Vert _{H^{\otimes
2q-2r}}^{2}\leq \Vert f\otimes _{r}f\Vert _{H^{\otimes 2q-2r}}^{2}$$ by the
triangle inequality.

\section{Proof of Proposition \ref{prop var}}\label{proofvar}

From Lemma \ref{lem chaos} we have, by the orthogonality property of chaotic projections and \paref{producthermite}
\begin{equation}\label{chainvar}
\begin{split}
\Var(D_\ell) &= \sum_{q=1}^{+\infty} \left(\frac{J_{2q+1}}{(2q+1)!}\right)^2 \int_{\mathbb S^d} \int_{\mathbb S^d}\E[H_{2q+1}(T_\ell(x)) H_{2q+1}(T_\ell(y))]\,dxdy\cr
&=\sum_{q=1}^{+\infty} \frac{J_{2q+1}^2}{(2q+1)!} \int_{\mathbb S^d} \int_{\mathbb S^d}G_{\ell;d}(\cos d(x,y))^{2q+1}\,dxdy.
\end{split}
\end{equation}
The isotropy property of the integrand function in the r.h.s. of \paref{chainvar} and  the choice of standard coordinates on the hypersphere allow to write
\begin{equation}\label{vaar}
\Var(D_\ell)=\sum_{q=1}^{+\infty} \frac{J_{2q+1}^2}{(2q+1)!} |\mathbb S^d||\mathbb S^{d-1}| \int_{0}^\pi G_{\ell;d}(\cos \vartheta)^{2q+1}(\sin \vartheta)^{d-1}\,d\vartheta,
\end{equation}
which coincides with \paref{var introtilde}. 
We are now ready to give the proof of Proposition \ref{prop var},  inspired by 
the proofs of \cite[Proposition 4.2, Theorem 1.2]{def}.

\begin{proof}[Proof of Proposition \ref{prop var}]
Let us bear in mind \paref{var intro}. In \cite[Proposition 2.2]{maudom} it has been proven that, as $\ell\to +\infty$,
\begin{equation}\label{limG}
\lim_{\ell\to +\infty} \ell^d \int_{0}^{\pi/2} G_{\ell;d}(\cos \vartheta)^{2q+1}(\sin \vartheta)^{d-1}\,d\vartheta= c_{2q+1;d},
\end{equation}
where $c_{2q+1;d}$ is given by
$
c_{2q+1;d}:=\int_{0}^{+\infty
} \widetilde J_{d}(\psi )^{2q+1}\psi ^{d-1}d\psi
$,
 $\widetilde J_d$ being defined as in \paref{tildej}. 
 Therefore, as anticipated in \S\ref{on the var} we would expect that, asymptotically,
\begin{equation}\label{guess}
\Var(D_\ell)\sim  \frac{C_d}{\ell^d},
\end{equation}
where the leading constant $C_d$ (uniquely depending on $d$) satisfies
\begin{equation}\label{guess1}
C_d = 2|\mathbb S^d||\mathbb S^{d-1}|\sum_{q=1}^{+\infty} \frac{J_{2q+1}^2}{(2q+1)!} c_{2q+1;d}.
\end{equation}
Before proving \paref{guess} that coincides with \paref{vareq},
let us check that $C_d>0$, assuming \paref{guess1} is true.
Actually, the r.h.s. of \paref{guess1} is a series of  nonnegative terms and from \cite[p. 217]{andrews} we have
\begin{equation*}
c_{3;d} = \left(2^{\frac{d}{2} - 1}\left (\frac{d}2-1 \right)!\right)^3
\frac{3^{\frac{d}{2} -\frac32}}{2^{3\left (\frac{d}2-1 \right)-1}\sqrt \pi\,
\Gamma \left ( \frac{d}{2} -\frac12   \right )}>0\ .
\end{equation*}
The previous argument proves also \paref{disug}. 

Moreover, assuming \paref{guess1} is true and keeping in mind \paref{tildej}, we get
\begin{equation}\label{chains}
\begin{split}
C_d  &= 2|\mathbb S^d||\mathbb S^{d-1}|\sum_{q=1}^{+\infty} \frac{J_{2q+1}^2}{(2q+1)!} 
\int_{0}^{+\infty
}\widetilde J_{d}(\psi)^{2q+1}\,\psi^{d-1}\,d\psi\cr
&= 2|\mathbb S^d||\mathbb S^{d-1}|\int_{0}^{+\infty}\sum_{q=1}^{+\infty} \frac{J_{2q+1}^2}{(2q+1)!} \,
\widetilde J_{d}(\psi)^{2q+1}\,\psi^{d-1}\, d\psi\cr
&=\frac{4}{\pi}|\mathbb S^d||\mathbb S^{d-1}|\int_{0}^{+\infty}
\left( \arcsin \left (\widetilde J_{d}(\psi)\right)              -  \widetilde J_{d}(\psi) \right)\psi^{d-1}\,d\psi.
\end{split}
\end{equation}
Actually, it is readily checked that the sequence 
$
\left(J_{2q+1}^2/(2q+1)!\right)_q
$
gives the coefficients in the Taylor expansion for the arcsin function. Actually, since $H_{2q}(0) = (-1)^q (2q-1)!!$ (where by $(2q-1)!!$ we mean the product of all positive odd integers $\le 2q-1$),
\begin{equation}\label{arcsincoeff}
\frac{J_{2q+1}^2}{(2q+1)!} = \frac{2}{\pi} \frac{(2q)!}{4^q (q!)^2 (2q+1)}.
\end{equation}
To justify the exchange
of the  integration and summation order in \paref{chains}, we consider, for $m\in\mathbb N$, $m>1$, the finite summation
$$
\sum_{q=1}^{m} \frac{J_{2q+1}^2}{(2q+1)!}  \int_0^{+\infty}
\widetilde J_{d}(\psi)^{2q+1}\psi^{d-1}\, d\psi
$$ and 
using the asymptotics
\begin{equation}\label{asymparcsin}
\frac{J_{2q+1}^2}{(2q+1)!} \sim \frac{c}{q^{3/2}}
\end{equation}
for some $c>0$ (Stirling's formula applied to \paref{arcsincoeff}), and the behavior of Bessel functions for large argument \cite[\S 1.7]{szego} to bound the contributions of tails, we take the limit $m\to +\infty$.

Let us now formally prove the asymptotic result for the variance \paref{guess}.
Note that, since the family $\left(G_{\ell;d}\right)_\ell$ is uniformly bounded by $1$, 
\begin{equation}\label{bellazio}
\begin{split}
& \sum_{q=m+1}^{+\infty} \frac{J_{2q+1}^2}{(2q+1)!}  \int_{0}^{\pi/2}
\left | G_{\ell;d}(\cos \theta) \right |^{2q+1}(\sin \theta)^{d-1}\,d\theta\cr
&\le   \sum_{q=m+1}^{+\infty} \frac{J_{2q+1}^2}{(2q+1)!}  \int_{0}^{\pi/2}
\left | G_{\ell;d}(\cos \theta) \right |^{5}(\sin \theta)^{d-1}\,d\theta
\ll \frac{1}{\ell^d} \sum_{q=m+1}^{+\infty} \frac{1}{q^{3/2}} \cr 
&\ll \frac{1}{\sqrt{m}\, \ell^d}\ .\end{split}
\end{equation}
Therefore, for $m=m(\ell)$ to be chosen
$$
\Var(D_\ell)=2|\mathbb S^d||\mathbb S^{d-1}| \sum_{q=1}^{m} \frac{J_{2q+1}^2}{(2q+1)!}  \int_{0}^{\pi/2}
G_{\ell;d}(\cos \theta)^{2q+1}(\sin \theta)^{d-1}\,d\theta + O\left (   \frac{1}{\sqrt{m}\, \ell^d} \right).
$$
From \paref{limG}, we can write
$$
\Var(D_\ell) =C_{d,m} \cdot \frac{1}{\ell^d} + o\left (\ell^{-d} \right) + O\left ( \frac{1}{\sqrt{m}\, \ell^d} \right ),
$$
where
$$
C_{d,m}:=2|\mathbb S^d||\mathbb S^{d-1}| \sum_{q=1}^{m} \frac{J_{2q+1}^2}{(2q+1)!} c_{2q+1;d}.
$$
Now since $C_{d,m}\to C_d$ as $m\to +\infty$, we can conclude.
\end{proof}

\section{Proof of Theorem \ref{th qclt} assuming Lemma \ref{lem}}\label{proofqclt}

Bearing in mind \paref{chaosexpintro}, let us set for simplicity of notation
$$
h_{\ell;2q+1,d}:=\int_{\mathbb S^d} H_{2q+1}(T_\ell(x))\,dx.
$$
For $m\in \N, m>1$ to be chosen later, we define the ``truncated" defect as 
$$
D_\ell^m:=\sum_{q=1}^{m-1} \frac{J_{2q+1}}{(2q+1)!}\,h_{\ell;2q+1,d}.
$$
We still have to introduce some more notation. As usual, $Z\sim \mathcal N(0,1)$, wherease $Z_{\ell,m}$ shall denote a centered Gaussian random variable with variance 
$\sigma^2_{\ell,m}:= \frac{\Var(D^m_\ell)}{\Var(D_\ell)}$. 
We are now ready  to give the proof of the quantitative CLT for the defect on the hypersphere in any dimension. The technique we adopt was used to prove rate of convergence to the Gaussian distribution in other circumstances (e.g. \cite{vale3,pham}). 

\begin{proof}[Proof of Theorem \ref{th qclt}]
By the triangle inequality 
\begin{equation*}
\begin{split}
&\text{d}_W \left( \frac{D_\ell}{\sqrt{\Var(D_\ell)}},  Z \right ) \cr
&\le \text{d}_W \left( \frac{D_\ell}{\sqrt{\Var(D_\ell)}},  \frac{D^m_\ell}{\sqrt{\Var(D_\ell)}}    \right ) + \text{d}_W \left( \frac{D^m_\ell}{\sqrt{\Var(D_\ell)}}, Z_{\ell,m}    \right ) + \text{d}_W \left( Z_{\ell,m} , Z   \right ).
\end{split}
\end{equation*}
Let us denote $I_1 := \text{d}_W \left( \frac{D_\ell}{\sqrt{\Var(D_\ell)}},  \frac{D^m_\ell}{\sqrt{\Var(D_\ell)}}    \right )$, $I_2:=\text{d}_W \left( \frac{D^m_\ell}{\sqrt{\Var(D_\ell)}}, Z_{\ell,m}    \right )$ and $I_3:=\text{d}_W \left( Z_{\ell,m} , Z   \right )$.

\smallskip

\emph{\large Bounding the contribution of $I_1$.}
By definition of Wasserstein distance \paref{wass} we have 
\begin{equation*}
\begin{split}
I_1^2 = \text{d}_W \left( \frac{D_\ell}{\sqrt{\Var(D_\ell)}},  \frac{D^m_\ell}{\sqrt{\Var(D_\ell)}}    \right )^2 &\le  \E\left [ \left ( \frac{D_\ell}{\sqrt{\Var(D_\ell)}} -  \frac{D^m_\ell}{\sqrt{\Var(D_\ell)}}          \right )^2           \right]      \cr
&=\frac{1}{\Var(D_\ell)}  \E\left [ \left ( D_\ell - D_\ell^m          \right )^2           \right]    \cr
& = \frac{1}{\Var(D_\ell)}
\sum_{q=m}^{+\infty}  \frac{J_{2q+1}^2}{((2q+1)!)^2}\Var\left( h_{\ell;2q+1,d}\right),
\end{split}
\end{equation*}
where the last equality still follows from the othogonality property of  chaotic projections.
From \paref{bellazio} and \paref{vareq}, one infers that
\begin{equation}\label{bound1}
\text{d}_W \left( \frac{D_\ell}{\sqrt{\Var(D_\ell)}},  \frac{D^m_\ell}{\sqrt{\Var(D_\ell)}}    \right ) =O\left (\frac{1}{^4\sqrt{m}} \right ).
\end{equation}

\smallskip

\emph{\large Bounding the contribution of $I_2$.} We will use some technical results involving Stein-Malliavin approximation methods \cite[Chapters 5, 6]{noupebook}. Let us set, from now on, $H:=L^2(\mathbb S^d)$.

From Proposition \ref{BIGnourdinpeccati} and \paref{casoparticolare} we deduce
\begin{equation}\label{eee}
\begin{split}
 &\text{d}_W \left( \frac{D^m_\ell}{\sqrt{\Var(D_\ell)}},  Z_{\ell,m}   \right ) \cr
&\le \sqrt{\frac{2}{\pi}}\sqrt{\frac{\Var(D_\ell)}{\Var(D_\ell^m)}}\E \left [ \left |\frac{\Var(D^m_\ell)}{\Var(D_\ell)} - \left \langle D  \frac{D^m_\ell}{\sqrt{\Var(D_\ell)}} , -DL^{-1}  \frac{D^m_\ell}{\sqrt{\Var(D_\ell)}} \right \rangle_H     \right |    \right ]\cr
&=\sqrt{\frac{2}{\pi}}\sqrt{\frac{\Var(D_\ell)}{\Var(D_\ell^m)}}\frac{1}{\Var(D_\ell)}\E \left [ \left |\Var(D^m_\ell) - \left \langle D  D^m_\ell , -DL^{-1}  D^m_\ell \right \rangle_H     \right |    \right ].
\end{split}
\end{equation}
Since
\begin{equation*}
\Var(D^m_\ell) = \sum_{q=1}^{m-1} \frac{J_{2q+1}^2}{((2q+1)!)^2}\Var(h_{\ell;2q+1,d}),
\end{equation*}
we can write for the r.h.s. of \paref{eee}
\begin{equation}\label{1}
\begin{split}
&\sqrt{\frac{2}{\pi}}\sqrt{\frac{\Var(D_\ell)}{\Var(D_\ell^m)}}\frac{1}{\Var(D_\ell)}\E \left [ \left |\Var(D^m_\ell) - \left \langle D  D^m_\ell , -DL^{-1}  D^m_\ell \right \rangle_H     \right |    \right ] \cr
&\le \sqrt{\frac{2}{\pi}}\sqrt{\frac{\Var(D_\ell)}{\Var(D_\ell^m)}}\frac{1}{\Var(D_\ell)}\sum_{q=1}^{m-1}\frac{J_{2q+1}^2}{((2q+1)!)^2} \E \left [ \left |\Var(h_{\ell;2q+1,d}) - \left \langle D  h_{\ell;2q+1,d} , -DL^{-1}  h_{\ell;2q+1,d} \right \rangle_H     \right |    \right ]\cr
&+\sqrt{\frac{2}{\pi}}\sqrt{\frac{\Var(D_\ell)}{\Var(D_\ell^m)}}\frac{1}{\Var(D_\ell)}\sum_{q=1}^{m-1}\frac{J_{2q+1}}{(2q+1)!} \sum_{q\ne p}\frac{J_{2p+1}}{(2p+1)!}\E \left [ \left | \left \langle D  h_{\ell;2q+1,d} , -DL^{-1}  h_{\ell;2p+1,d} \right \rangle_H     \right |    \right ] \cr
&\le \sqrt{\frac{2}{\pi}}\sqrt{\frac{\Var(D_\ell)}{\Var(D_\ell^m)}}\frac{1}{\Var(D_\ell)}\sum_{q=1}^{m-1}\frac{J_{2q+1}^2}{((2q+1)!)^2} \sqrt{\Var \left ( \left \langle D  h_{\ell;2q+1,d} , -DL^{-1}  h_{\ell;2q+1,d} \right \rangle_H         \right )} \cr
&+\sqrt{\frac{2}{\pi}}\sqrt{\frac{\Var(D_\ell)}{\Var(D_\ell^m)}}\frac{1}{\Var(D_\ell)}\sum_{q=1}^{m-1}\frac{J_{2q+1}}{(2q+1)!} \sum_{q\ne p}\frac{J_{2p+1}}{(2p+1)!}\sqrt{\E \left [  \left \langle D  h_{\ell;2q+1,d} , -DL^{-1}  h_{\ell;2p+1,d} \right \rangle_H ^2 \right ]},
\end{split}
\end{equation}
where for the last step we used \cite[Theorem 2.9.1]{noupebook} and  Cauchy-Schwarz inequality. 
Now, Lemma \ref{comevale} which is collected in the Appendix \ref{valeapp} and whose proof relies in particular on \cite[Proposition 4.2, Proposition 4.3]{maudom} and Lemma \ref{lem}, allows one to write, for some $C>0$,
\begin{equation*}
\begin{split}
&\Var \left (\left \langle D  h_{\ell;2q+1,d} , -DL^{-1}  h_{\ell;2q+1,d} \right \rangle_H         \right )\le C (2q+1)^2 ((2q)!)^2 3^{4q} R_{\ell;d},\cr
&\E \left [  \left \langle D  h_{\ell;2q+1,d} , -DL^{-1}  h_{\ell;2p+1,d} \right \rangle_H ^2 \right ]
\le C (2q+1)^2(2q)! (2p)! 3^{2q+2p} R_{\ell;d},
\end{split}
\end{equation*}
where 
\begin{equation*}
R_{\ell;2} := \frac{\log\ell}{\ell^{9/2}}\quad \text{ and for } d>2 \quad  R_{\ell;d}:=\frac{1}{\ell^{2d +(d-1)/2}}.
\end{equation*}
Therefore, Lemma \ref{comevale}, \paref{eee} and \paref{1} give
\begin{equation*}
\begin{split}
 &\text{d}_W \left( \frac{D^m_\ell}{\sqrt{\Var(D_\ell)}},  Z_{\ell,m}   \right ) \cr
&\le \sqrt{\frac{2}{\pi}}\sqrt{\frac{\Var(D_\ell)}{\Var(D_\ell^m)}}\frac{\sqrt{R_{\ell;d}}}{\Var(D_\ell)}\sum_{q=1}^{m-1}\frac{J_{2q+1}^2}{((2q+1)!)^2} \sqrt{C (2q+1)^2 ((2q)!)^2 3^{4q} } \cr
&+\sqrt{\frac{2}{\pi}}\sqrt{\frac{\Var(D_\ell)}{\Var(D_\ell^m)}}\frac{\sqrt{R_{\ell;d}}}{\Var(D_\ell)}\sum_{q=1}^{m-1}\frac{J_{2q+1}}{(2q+1)!} \sum_{q\ne p}\frac{J_{2p+1}}{(2p+1)!}\sqrt{C (2q+1)^2(2q)! (2p)! 3^{2q+2p} }.
\end{split}
\end{equation*}
Now 
\begin{equation*}
\sum_{q=1}^{m-1}\frac{J_{2q+1}^2}{((2q+1)!)^2} \sqrt{ (2q+1)^2 ((2q)!)^2 3^{4q}}=\sum_{q=1}^{m-1}\frac{J_{2q+1}^2}{(2q+1)!} 3^{2q}
\end{equation*}
and by \paref{asymparcsin} one deduces 
\begin{equation}
\sum_{q=1}^{m-1}\frac{J_{2q+1}^2}{((2q+1)!)^2} \sqrt{ (2q+1)^2 ((2q)!)^2 3^{4q}}\le C 3^{2m}
\end{equation}
and analogously
\begin{equation}
\sum_{q=1}^{m-1}\frac{J_{2q+1}}{(2q+1)!} \sum_{q\ne p}\frac{J_{2p+1}}{(2p+1)!}\sqrt{(2q+1)^2(2q)! (2p)! 3^{2q+2p} }\le C 3^{2m},
\end{equation}
for some positive constant $C$. Finally, thanks to Proposition \ref{prop var} and since $\frac{1}{\sigma^2_{\ell;m}}\le \frac{C}{1 - m^{-1/2}}$ (by \paref{bellazio}), we can write
\begin{equation}\label{bound2}
\begin{split}
 \text{d}_W \left( \frac{D^m_\ell}{\sqrt{\Var(D_\ell)}}, Z_{\ell,m}   \right )
& \le C  3^{2m}\sqrt{\frac{1}{1 - m^{-1/2}}} \ell^d\sqrt{R_{\ell;d}}.
\end{split}
\end{equation}

\smallskip 

\emph{\large Bounding the contribution of $I_3$.}
Proposition 3.6.1 in \cite{noupebook} and \paref{bellazio} give 
\begin{equation}\label{bound3}
 \text{d}_W \left( Z_{\ell,m} , Z   \right )\le C \left| \frac{\Var(D_\ell - D_\ell^m)}{\Var(D_\ell)}\right |\ll \frac{1}{\sqrt m}.
\end{equation}

\smallskip

\emph{\large Optimazing on $m$.}
Summing up the three bounds \paref{bound1}, \paref{bound2} and \paref{bound3} with the choice of the speed 
$
m\asymp \log \ell^\alpha
$, for some $\alpha\in (0,1/4)$, 
it is easy to note that, in any dimension, the dominant term is given by \paref{bound1}. This concludes the proof. 
\end{proof}

\section{Proofs of Lemma \ref{lem} and Lemma \ref{lemCG}}\label{345}

We shall use sometimes in this section the shorthand notation
\begin{equation*}
\sum_{m}:=\sum_{m=1}^{n_{\ell;d}}\qquad\text{and}\qquad \sum_{m_1,m_2}:=\sum_{m_1,m_2=1}^{n_{\ell;d}}.
\end{equation*}
Moreover we will drop the dependence of $d$ in \paref{defcg}, for brevity. 

\begin{proof}[Proof of Lemma \ref{lem} assuming Lemma \ref{lemCG}]
Let us study the following multiple integral which gives the major contribution to the $4$-th cumulant of $\int_{\mathbb S^d} H_3(T_\ell(x))\,dx$, as already stated in \S\ref{on asymp}:
\begin{equation*}\label{4cum}
\int_{(\mathbb S^d)^4} G_{\ell;d}(\cos d(w,z))G_{\ell;d}(\cos d(w,w'))^2G_{\ell;d}(\cos d(w',z'))G_{\ell;d}(\cos d(z',z))^2\,dw dz dw' dz'.
\end{equation*}
The addition formula for Gegenbauer polynomials \paref{addformula} allows one to write
\begin{equation}\label{h1}
\begin{split}
&\int_{\mathbb S^d} G_{\ell;d}(\cos d(w,z))G_{\ell;d}(\cos d(w,w'))^2\,dw  \cr
&= \left(\frac{|\mathbb S^d|}{n_{\ell;d}}\right)^3 \sum_{m_1,m_2',m_3'=1}^{n_{\ell;d}} Y_{\ell,m_1}(z)Y_{\ell,m_2'}(w') Y_{\ell,m_3'}(w')  \int_{\mathbb S^d} Y_{\ell,m_1}(w)  Y_{\ell,m_2'}(w) Y_{\ell,m_3'}(w)\,dw\cr
&=\left(\frac{|\mathbb S^d|}{n_{\ell;d}}\right)^3 \sum_{m_1,m_2',m_3'=1}^{n_{\ell;d}} Y_{\ell,m_1}(z)Y_{\ell,m_2'}(w') Y_{\ell,m_3'}(w') {\mathcal G}^{\ell,m_1}_{\ell,m_2',\ell,m_3'}
\end{split}
\end{equation}
where ${\mathcal G}^{\ell,m_1}_{\ell,m_2',\ell,m_3'}$ has been defined in \paref{defcg}. 
The same argument applied in order to deduce \paref{h1} gives
\begin{equation}\label{h2}
\begin{split}
&\int_{\mathbb S^d} G_{\ell;d}(\cos d(w',z'))G_{\ell;d}(\cos d(z',z))^2\,dz'\cr
&= \left(\frac{|\mathbb S^d|}{n_{\ell;d}}\right)^3 \sum_{m_1',m_2,m_3=1}^{n_{\ell;d}} Y_{\ell,m_1'}(w')Y_{\ell,m_2}(z) Y_{\ell,m_3}(z) {\mathcal G}^{\ell,m_1'}_{\ell,m_2,\ell,m_3}.
\end{split}
\end{equation}
We are thus left with
\begin{equation}\label{core}
\begin{split}
&\int_{(\mathbb S^d)^4} G_{\ell;d}(\cos d(w,z))G_{\ell;d}(\cos d(w,w'))^2G_{\ell;d}(\cos d(w',z'))G_{\ell;d}(\cos d(z',z))^2\,dw dz dw' dz'\cr
&=\left(\frac{|\mathbb S^d|}{n_{\ell;d}}\right)^6 \sum_{m_1,m_2,m_3,m_1',m_2',m_3'=1}^{n_{\ell;d}}{\mathcal G}^{\ell,m_1}_{\ell,m_2',\ell,m_3'}{\mathcal G}^{\ell,m_1'}_{\ell,m_2,\ell,m_3}\times \cr
&\times \int_{(\mathbb S^d)^2}  Y_{\ell,m_1'}(w')Y_{\ell,m_2}(z) Y_{\ell,m_3}(z) Y_{\ell,m_1}(z)Y_{\ell,m_2'}(w') Y_{\ell,m_3'}(w')\,dzdw'\cr
&=\left(\frac{|\mathbb S^d|}{n_{\ell;d}}\right)^6 \sum_{m_1,m_2,m_3,m_1',m_2',m_3'=1}^{n_{\ell;d}}{\mathcal G}^{\ell,m_1}_{\ell,m_2',\ell,m_3'}{\mathcal G}^{\ell,m_1'}_{\ell,m_2,\ell,m_3}{\mathcal G}^{\ell,m_1}_{\ell,m_2,\ell,m_3}{\mathcal G}^{\ell,m_1'}_{\ell,m_2',\ell,m_3'}.
\end{split}
\end{equation}

 Lemma \ref{lemCG} states that
\begin{equation}\label{ciao}
 \sum_{m_2,m_3}{\mathcal G}^{\ell,m_1'}_{\ell,m_2,\ell,m_3}{\mathcal G}^{\ell,m_1}_{\ell,m_2,\ell,m_3}= g_{\ell;d}\,\delta_{m_1}^{m_1'},
\end{equation}
where $$g_{\ell;d} := \frac{(n_{\ell;d})^2}{|\mathbb S^d|} \frac{|\mathbb S^{d-1}| }{|\mathbb S^{d}|}\int_{-1}^1 G_{\ell;d}(t)^3 \left(\sqrt{1-t^2}\right)^{d-2}\,dt.$$ 
Plugging \paref{ciao} into \paref{core} and applying once more Lemma \ref{lemCG} we find 

\begin{equation}\label{moltobello}
\begin{split}
&\left(\frac{|\mathbb S^d|}{n_{\ell;d}}\right)^6 \sum_{m_1,m_2,m_3,m_1',m_2',m_3'=1}^{n_{\ell;d}}{\mathcal G}^{\ell,m_1}_{\ell,m_2',\ell,m_3'}{\mathcal G}^{\ell,m_1'}_{\ell,m_2,\ell,m_3}{\mathcal G}^{\ell,m_1}_{\ell,m_2,\ell,m_3}{\mathcal G}^{\ell,m_1'}_{\ell,m_2',\ell,m_3'}\cr
&=\left(\frac{|\mathbb S^d|}{n_{\ell;d}}\right)^6 g_{\ell;d}\sum_{m_1,m_2',m_3'}{\mathcal G}^{\ell,m_1}_{\ell,m_2',\ell,m_3'}{\mathcal G}^{\ell,m_1}_{\ell,m_2',\ell,m_3'}\cr
&=\left(\frac{|\mathbb S^d|}{n_{\ell;d}}\right)^6  g_{\ell;d}\sum_{m_1}\underbrace{\sum_{m_2',m_3'}{\mathcal G}^{\ell,m_1}_{\ell,m_2',\ell,m_3'}{\mathcal G}^{\ell,m_1}_{\ell,m_2',\ell,m_3'}}_{=g_{\ell;d}}\cr
&=\left(\frac{|\mathbb S^d|}{n_{\ell;d}}\right)^6 \left (g_{\ell;d}\right )^2 n_{\ell;d} =|\mathbb S^d|^6 \left(\frac{1}{n_{\ell;d}}\right)^5 \left (g_{\ell;d}\right )^2.
\end{split}
\end{equation}
From \cite[Proposition 2.2]{maudom}, it is readily checked that $g_{\ell;d}\asymp \ell^{d-2}$. Hence \paref{moltobello} gives
\begin{equation*}
\text{cum}_4\left( \int_{\mathbb S^d} H_3(T_\ell(x))\,dx  \right)\ll  \frac{1}{\ell^{3d-1}}
\end{equation*}
and from Proposition \ref{BIGnourdinpeccati} and \paref{casoparticolare}, bearing in mind \paref{limG}, we have 
\begin{equation*}
\begin{split}
\text{d}_{\mathcal D}\left( \frac{\int_{\mathbb S^d} H_3(T_\ell(x))\,dx}{\sqrt{\Var\left( \int_{\mathbb S^d} H_3(T_\ell(x))\,dx  \right)}}, Z \right) &= O\left( \sqrt{\frac{\text{cum}_4\left( \int_{\mathbb S^d} H_3(T_\ell(x))\,dx  \right)}{\Var\left( \int_{\mathbb S^d} H_3(T_\ell(x))\,dx  \right)^2}}   \right)=O\left( \sqrt{\frac{1}{\ell^{d-1}}}   \right).
\end{split}
\end{equation*}
\end{proof}

\subsection{Proof of Lemma \ref{lemCG}}\label{CG}
\begin{proof}[Proof]
Since by definition 
\begin{equation*}
{\mathcal G}^{\ell,M}_{\ell,m_1,\ell,m_2} = \int_{\mathbb S^d}  Y_{\ell,m_1}(x)Y_{\ell,m_2}(x)Y_{\ell,M}(x)\,dx, 
\end{equation*}
we can write
\begin{equation}\label{chainbella}
\begin{split}
&\sum_{m_1,m_2} {\mathcal G}^{\ell,M}_{\ell,m_1,\ell,m_2}{\mathcal G}^{\ell,M'}_{\ell,m_1,\ell,m_2}\cr
&= \sum_{m_1,m_2} \int_{\mathbb S^d}  Y_{\ell,m_1}(x)Y_{\ell,m_2}(x)Y_{\ell,M}(x)\,dx \int_{\mathbb S^d}  Y_{\ell,m_1}(y)Y_{\ell,m_2}(y)Y_{\ell,M'}(y)\,dy\cr
&=\int_{\mathbb S^d} \int_{\mathbb S^d} dx dy  Y_{\ell,M}(x) Y_{\ell,M'}(y)\sum_{m_1} Y_{\ell,m_1}(x) Y_{\ell,m_1}(y)\sum_{m_1}Y_{\ell,m_2}(x) Y_{\ell,m_2}(y)\cr
&=\int_{\mathbb S^d} \int_{\mathbb S^d} dx dy  Y_{\ell,M}(x) Y_{\ell,M'}(y)\left( \frac{n_{\ell;d}}{|\mathbb S^d|}\right)^2 G_{\ell;d}(\cos d(x,y))^2,
\end{split}
\end{equation}
where in the last equality we used twice the addition formula \paref{addformula} for Gegenabuer polynomials. The family $\left( \sqrt{|\mathbb S^{d-1}| n_{\ell;d}/|\mathbb S^{d}|}\,G_{\ell;d}\right)_\ell$ being orthonormal on $[-1,1]$, the following equality holds
\begin{equation}\label{polyn}
G_{\ell;d}(t)^2=\sum_{j=0}^{2\ell} \gamma_j G_{j;d}(t),\qquad t\in [-1,1],
\end{equation}
where the coefficients $\gamma_j := \gamma_j(\ell;d)$ are given by 
\begin{equation*}
\gamma_j = n_{\ell;d}\frac{|\mathbb S^{d-1}| }{|\mathbb S^{d}|} \int_{-1}^1 G_{\ell;d}(t)^2G_{j;d}(t) (\sqrt{1-t^2})^{d-2}\,dt.
\end{equation*}
Therefore, plugging \paref{polyn} into \paref{chainbella}, by the orthormality property of hyperspherical harmonics one deduces 

\begin{equation*}
\begin{split}
\sum_{m_1,m_2} {\mathcal G}^{\ell,M}_{\ell,m_1,\ell,m_2}{\mathcal G}^{\ell,M'}_{\ell,m_1,\ell,m_2}&=\int_{\mathbb S^d} \int_{\mathbb S^d} dx dy  Y_{\ell,M}(x) Y_{\ell,M'}(y)\left( \frac{n_{\ell;d}}{|\mathbb S^d|}\right)^2 \sum_{j} \gamma_j G_{j;d}(\cos d(x,y))\cr
&= \int_{\mathbb S^d} \int_{\mathbb S^d} dx dy  Y_{\ell,M}(x) Y_{\ell,M'}(y)\left( \frac{n_{\ell;d}}{|\mathbb S^d|}\right)^2 \sum_{j} \gamma_j \sum_{k} \frac{|\mathbb S^d|}{n_{j;d}} Y_{j,k}(x)Y_{j,k}(y)\cr
&=\left( \frac{n_{\ell;d}}{|\mathbb S^d|}\right)^2\sum_{j} \gamma_j \frac{|\mathbb S^d|}{n_{j;d}}\sum_{k}\underbrace{\int_{\mathbb S^d}    Y_{\ell,M}(x)   Y_{j,k}(x)\,dx}_{=\delta_\ell^j \delta_M^k} \underbrace{\int_{\mathbb S^d} Y_{\ell,M'}(y) Y_{j,k}(y)\,dy}_{=\delta_{\ell}^j \delta_{M'}^k}\cr
&=\frac{n_{\ell;d}}{|\mathbb S^d|} \gamma_\ell\, \delta_M^{M'}.
\end{split}
\end{equation*}
This concludes the proof.
\end{proof}

\section{Appendix}\label{app}

\subsection{Proof of Lemma \ref{lem chaos}}\label{appchaos}

\begin{proof} 
By \paref{defect}, one deduces that $|D_\ell|\le |\mathbb S^d|$ a.s. and hence $D_\ell\in L^2(\P)$. Recall that we can write 
$$
D_\ell = 2\int_{\mathbb S^d} 1_{(0,+\infty)}(T_\ell(x))\,dx - |\mathbb S^d|.
$$
The chaotic expansion \S \ref{defchaos} of the indicator function $1_{(0,+\infty)}$ is  given by (see e.g. \cite{Nonlin} and the references therein)
$$
1_{(0,+\infty)}(\cdot ) = \frac12 + \sum_{q\ge 0} \frac{\phi(0)H_{2q}(0)}{(2q+1)!} H_{2q+1}(\cdot ),
$$
where $\phi $ still denotes the p.d.f. of the standard Gaussian law and $(H_k)_{k\ge 0}$ the sequence of Hermite polynomials. Hence in particular, 
\begin{equation}\label{converg}
\sum_{q\ge 0} \frac{(\phi(0)H_{2q}(0))^2}{(2q+1)!} = \Phi(0)(1-\Phi(0))< +\infty,
\end{equation}
$\Phi$ still denoting the cumulative distribution function of a standard Gaussian random variable. 
Actually, it is easy to check that, for $Z\sim \mathcal N(0,1)$, $\E[1_{(0,+\infty)}(Z)] = 1/2$, whereas for $k\ge 1$
\begin{equation*}
\begin{split}
\E[1_{(0,+\infty)}(Z)H_{k}(Z)]&=\int_0^{+\infty} (-1)^k \phi^{-1}(t)\frac{d^k\phi}{dt^k}(t)\phi(t)\,dt\cr
&= (-1)^k \frac{d^{k-1}\phi}{dt^{k-1}}(t)   \Big|_0^{+\infty} = -\phi(t) H_{k-1}(t)   \Big|_0^{+\infty}  \cr
&=\phi(0) H_{k-1}(0),
\end{split}
\end{equation*}
which vanishes if $k$ is even. 
For $m\in \mathbb N$, $m>0$, let us consider the random variable
\begin{equation*}\label{ok}
\begin{split}
U_\ell^m &:= 2\int_{\mathbb S^d} \left( \frac12 + \sum_{q= 0}^m  \frac{\phi(0)H_{2q}(0)}{(2q+1)!} H_{2q+1}(T_\ell(x) ) \right)dx - |\mathbb S^d|\cr
&=\sum_{q= 1}^m  2\frac{\phi(0)H_{2q}(0)}{(2q+1)!}\int_{\mathbb S^d}  H_{2q+1}(T_\ell(x) )\,dx,
\end{split}
\end{equation*}
where the sum starts from $q=1$ since $H_1(t) = t$ and hyperspherical harmonics have zero mean over $\mathbb S^d$. Let us set moreover
$$
J_{2q+1} := 2\phi(0)H_{2q}(0). 
$$
In what follows, we shall show that the sequence of random variables $\left(U_\ell^m\right)_m$ is a Cauchy sequence in $L^2(\P)$. By the orthogonality property of chaotic projections and \paref{producthermite}, we have for $m,n \in \mathbb N$, $n,m>0$
\begin{equation*}
\begin{split}
\E[(U_\ell^m - U_\ell^{m+n})^2] = \sum_{q=m+1}^{m+n} \frac{J_{2q+1}^2}{(2q+1)!}\int_{(\mathbb S^d)^2} G_{\ell;d}(\cos d(x,y))^{2q+1}\,dxdy.
\end{split}
\end{equation*}
Now, since Gegenbauer polynomials are uniformly bounded by $1$, we have 
\begin{equation*}
\begin{split}
\E[(U_\ell^m - U_\ell^{m+n})^2]&\le |\mathbb S^d|^2\sum_{q=m+1}^{m+n} \frac{J_{2q+1}^2}{(2q+1)!};
\end{split}
\end{equation*}
hence \paref{converg} allows to conclude the proof. 
\end{proof}

\subsection{Some useful estimates}\label{valeapp}
Let us denote $H:=L^2(\mathbb S^d)$. 
\begin{lemma}\label{comevale}
There exists $C>0$ such that for integers $q,p\ge 1$, $q\le p$,
\begin{equation}\label{pesante}
\begin{split}
&\Var \left (\left \langle D  h_{\ell;2q+1,d} , -DL^{-1}  h_{\ell;2q+1,d} \right \rangle_H         \right )\le C (2q+1)^2 ((2q)!)^2 3^{4q} R_{\ell;d},\cr
&\E \left [  \left \langle D  h_{\ell;2q+1,d} , -DL^{-1}  h_{\ell;2p+1,d} \right \rangle_H ^2 \right ]
\le C (2q+1)^2(2q)! (2p)! 3^{2q+2p} R_{\ell;d},
\end{split}
\end{equation}
where 
\begin{equation*}
R_{\ell;2} := \frac{\log\ell}{\ell^{9/2}}\quad \text{ and for } d>2 \quad  R_{\ell;d}:=\frac{1}{\ell^{2d +(d-1)/2}}.
\end{equation*}
\end{lemma}

\begin{proof}
Recall that $h_{\ell;2q+1,d}$ can be expressed as a multiple Wiener-It\^o integral of order $q$ (see \S \ref{multiple})
\begin{equation*}
h_{\ell;2q+1,d} \mathop{=}^{\mathcal L} \int_{(\mathbb S^d)^q} g_{\ell;2q+1,d}(y_1,y_2,\dots,y_q)\,dW(y_1)dW(y_2)\dots dW(y_q)=:I_q(g_{\ell;2q+1,d}),
\end{equation*}
where the function $g_{\ell;2q+1,d}$ is given by 
\begin{equation*}
g_{\ell;2q+1,d}(y_1,y_2,\dots,y_q) := \int_{\mathbb S^d} \left( \frac{n_{\ell;d}}{|\mathbb S^d|}  \right)^{q/2} G_{\ell;d}(\cos d(x,y_1)) \cdots G_{\ell;d}(\cos d(x,y_q))\,dx.
\end{equation*}
Similar arguments as those in the proof of \cite[Lemma 6.1]{vale3} allow one to have, for integers $p,q\ge 1$, the following new estimates 
\begin{equation}\label{dafare1}
\begin{split}
&\Var \left (\left \langle D  h_{\ell;2q+1,d} , -DL^{-1}  h_{\ell;2q+1,d} \right \rangle_H         \right )\cr
&\le (2q+1)^2 \sum_{r=1}^{2q} ((r-1)!)^2 { 2q \choose r-1}^4 (2(2q+1)-2r)!\| g_{\ell;2q+1,d}  \otimes_r g_{\ell;2q+1,d}\|^2_{H^{\otimes 2(2q+1)-2r}},
\end{split}
\end{equation}
and moreover for $q\le p$ 
\begin{equation}\label{dafare}
\begin{split}
&\E \left [  \left \langle D  h_{\ell;2q+1,d} , -DL^{-1}  h_{\ell;2p+1,d} \right \rangle_H ^2 \right ] \cr
&= (2q+1)^2 \sum_{r=1}^{2q+1} ((r-1)!)^2 { 2q \choose r-1}^2 { 2p \choose r-1}^2 (2q+2p + 2-2r)!\| g_{\ell;2q+1,d} \widetilde \otimes_r g_{\ell;2p+1,d}\|^2_{H^{\otimes n}} \cr
&\le (2q+1)^2 \sum_{r=1}^{2q+1} ((r-1)!)^2 { 2q \choose r-1}^2 { 2p \choose r-1}^2 (2q+2p + 2-2r)!\| g_{\ell;2q+1,d}  \otimes_r g_{\ell;2p+1,d}\|^2_{H^{\otimes n}},
\end{split}
\end{equation}
where $n:=2q+2p+2-2r$ for notational simplicity. 
Now from \cite[Proposition 4.1]{maudom} we know the explicit formula for the norm of contractions: for $q\le p$
\begin{equation*}
\begin{split}
\| g_{\ell;2q+1,d}  \otimes_r g_{\ell;2p+1,d}\|^2_{H^{\otimes n}}
=\int_{(\mathbb S^d)^4}&G_{\ell;d}(\cos d(x_1,x_2))^r G_{\ell;d}(\cos d(x_2,x_3))^{2q+1-r} \times\cr
&\times G_{\ell;d}(\cos d(x_3,x_4))^r G_{\ell;d}(\cos d(x_4,x_1))^{2q+1-r}d\underline{x},
\end{split}
\end{equation*}
where $d\underline{x}:=dx_1dx_2dx_3dx_4$. Thanks to \cite[Proposition 4.2, Proposition 4.3]{maudom} (for $q\ge 2$) and Lemma \ref{lem} (for $q=1$) we have, as $\ell\to +\infty$,
\begin{equation}\label{contr}
\| g_{\ell;2q+1,2}  \otimes_r g_{\ell;2p+1,2}\|^2_{H^{\otimes n}} = O\left(R_{\ell;d}  \right),
\end{equation}
where $R_{\ell;2} = \log\ell / \ell^{9/2}$ and $R_{\ell;d}=1/\ell^{2d +(d-1)/2}$ for $d>2$. Note that $O$' notation is independent of $q$ and $p$. 

As stated in \cite[(6.1),(6.2)]{vale3}, the following inequalities hold
\begin{equation}\label{facile}
\begin{split}
&\sum_{r=1}^{2q} ((r-1)!)^2 { 2q \choose r-1}^4 (2(2q+1)-2r)!\le ((2q)!)^2 3^{4q},\cr
&\sum_{r=1}^{2q+1} ((r-1)!)^2 { 2q \choose r-1}^2 { 2p \choose r-1}^2 (2q+2p + 2-2r)!\le (2q)! (2p)! 3^{2q+2p}.
\end{split}
\end{equation}
Plugging \paref{facile} and \paref{contr} into \paref{dafare1} and \paref{dafare}, one infers \paref{pesante}.
\end{proof}


\end{document}